\documentclass[english,11pt, oneside, a4paper]{amsart}

\usepackage{amsfonts,amsmath,amsthm,amssymb,amscd,latexsym,enumerate,etoolbox,mathrsfs,comment, graphicx}
\usepackage[all]{xy}
\usepackage{tikz-cd}
\usepackage{float}
\usepackage[left=2.5cm,right=2.5cm,top=2cm,bottom=2.5cm,bindingoffset=5mm]{geometry}
\usepackage[T1]{fontenc} 
\usepackage[utf8]{inputenc}
\usepackage{mathtools}
\usepackage{hyperref}
\usepackage{theoremref}
\usepackage{quiver}
\usepackage{tikz}
\usepackage{standalone} 
\usepackage{tikzscale} 
\usepackage{enumitem} 

\usepackage[normalem]{ulem} 

\usetikzlibrary{backgrounds,fit,positioning}
\usepackage{biblatex} 
\usepackage[english]{babel}

\newtheorem{theorem}{Theorem}[section]
\newtheorem{proposition}[theorem]{Proposition}
\newtheorem{lemma}[theorem]{Lemma}
\newtheorem{corollary}[theorem]{Corollary}
\theoremstyle{definition}
\newtheorem{remark}[theorem]{Remark}
\newtheorem{definition}[theorem]{Definition}
\newtheorem{example}[theorem]{Example}

\newcommand{\mb}[1]{\mathbb{#1}}

\newcommand{\R}{\mb R}
\newcommand{\Z}{\mb Z}
\newcommand{\N}{\mb N}
\newcommand{\C}{\mb C}
\newcommand{\V}{\mathcal{V}}

\newcommand{\E}{\mathcal{E}}

\newcommand{\calL}{\mathcal{L}}

\newcommand{\K}{\mathcal{K}}
\newcommand{\F}{\mathcal{F}}

\newcommand{\T}{\mathbb{T}}

\newcommand{\calU}{\mathcal{U}}

\DeclareMathOperator{\im}{Im}

\DeclareMathOperator{\Ad}{Ad}

\DeclareMathOperator{\rank}{rank}

\DeclareMathOperator{\aut}{Aut}

\DeclareMathOperator{\id}{id}

\title{Twisted topological correspondences and Cartan subalgebras}
\author{Aaron Kettner}
\date{\today}
\thanks{Funded by GA\v{C}R project GF22-07833K and \mbox{RVO: 67985840}. Part of this work was carried out while funded by GA\v{C}R project 20-17488Y. This work has been supported by Charles University Research Centre
program No. UNCE/24/SCI/022, and by the Charles University project SVV-2023-260721. The author is currently funded by GA\v{C}R project G25-15403K.}
\begin{document}
\maketitle

\begin{abstract}
    We introduce twisted topological correspondences, which generalize both Katsura's topological correspondences as well as the twisted topological graphs introduced by Li. We show that, up to isomorphism, they are in bijection with certain principal bundles. This makes it possible to study topological correspondences using the machinery of principal and fiber bundles. We show how to associate a $C^*$-correspondence to a twisted topological correspondence, and give two different characterizations of the $C^*$-correspondences arising that way. The first one is the existence of an atlas of the vector bundle associated to the $C^*$-correspondence whose transition functions take values in a certain subgroup of the unitary group $U(n)$, and which is in some sense compatible with the left action. The other characterization is in terms of Cartan subalgebras in the compact operators on the $C^*$-correspondence. We use our findings to prove rigidity results of the $C^*$-correspondences associated to twisted topological correspondences. 
\end{abstract}

\tableofcontents

\section{Introduction}

 In the modern theory of $C^*$-algebras the notion of a $C^*$-correspondence, which is a kind of bimodule whose coefficient algebras are $C^*$-algebras, is of fundamental importance. They are particularly useful for constructing interesting examples of $C^*$-algebras, via the construction of the Cuntz--Pimsner algebra from \cite{Pimsner:1997} and \cite{Katsura:2004}. Indeed, a lot of the classical ways to construct $C^*$-algebras, like crossed products by the integers or graph $C^*$-algebras, can be studied in a unified framework using $C^*$-correspondences and their Cuntz--Pimsner algebras. 

$C^*$-correspondences can be regarded as a notion of arrows between $C^*$-algebras that are more general than the usual $\ast$-homomorphisms \cite{Meyer:2012}. Since $C^*$-algebras are interpreted as noncommutative topological spaces by Gelfand duality, it is natural to consider correspondences between topological spaces. Such a notion was defined by Katsura in \cite{katsura:2004_2}, and called \emph{topological correspondence}. A topological correspondence between two locally compact Hausdorff spaces $X$ and $Y$ consists of a locally compact Hausdorff space $Z$, together with a local homeomorphism $s:Z\to X$ and a continuous map $r:Z\to Y$. In this way, topological correspondences generalize continuous maps between spaces, as well as discrete graphs. In the special case that $X$ and $Y$ are the same space, one obtains a topological graph as defined in \cite{katsura:2004_2}. In \cite{Li:2017} topological graphs were twisted by a complex line bundle, generalizing the twisted local homeomorphisms from \cite{deaconu:2001}. One can associate a $C^*$-correspondence to a topological correspondence or (twisted) topological graph, and then obtain a $C^*$-algebra as the Cuntz--Pimsner algebra of that $C^*$-correspondence. The $C^*$-algebras obtained this way include graph $C^*$-algebras and crossed products of commutative $C^*$-algebras by the integers. In a celebrated result \cite{Katsura:2008}, Katsura showed that topological graph $C^*$-algebras realise all UCT Kirchberg algebras, meaning all purely infinite simple separable nuclear $C^*$-algebras satisfying the Universal Coefficient Theorem (UCT).

A natural question is how to characterize which $C^*$-correspondences can be obtained from (twisted) topological graphs or correspondences. This is also very useful when trying to understand Cuntz--Pimsner algebras. Knowing that a given $C^*$-correspondence arises from a topological graph grants access to a wide range of results relating the structure of the topological graph to properties of its $C^*$-algebra. This includes conditions for simplicity and pure infiniteness \cite{Katsura:2008} as well as a description of the tracial state space \cite{Schafhauser:2016}. Furthermore, every twisted topological graph $C^*$-algebra has an underlying twisted étale groupoid \cite{KumjianLi:2017}. 

In \cite{frausino2023} the authors showed that a finitely generated and projective Hilbert $C(X)$-module, where $X$ is a compact Hausdorff space, can arise as the right module associated to a topological graph if and only if the associated vector bundle has an atlas whose transition functions take values in the permutation matrices. This result bears a curious similarity to one of Li and Renault \cite{LiRenault:2019}. They considered Cartan subalgebras in homogeneous $C^*$-algebras. A Cartan subalgebra $D$ of a $C^*$-algebra $A$ is a maximal abelian regular subalgebra such that there is a unique conditional expectation from $A$ to $D$. They were introduced by Renault in \cite{renault:2008} to give an algebraic characterization of twisted groupoid $C^*$-algebras, building on work of Kumjian \cite{kumjian:1986}. Recently Cartan subalgebras have attracted a lot of attention in the context of the Elliott classification program, since it has been shown that they are related to the question of whether any separable nuclear $C^*$-algebra satisfies the UCT \cite{BarlakLi:2017}, \cite{BarlakLi:2020}.

An $n$-homogeneous $C^*$-algebra can be described as a fiber bundle with fiber the $n\times n$-matrices, $M_n$, and structure group the automorphism group $\aut(M_n)$. It was shown in \cite{LiRenault:2019} that an $n$-homogeneous $C^*$-algebra has a Cartan subalgebra if and only if the principal bundle associated to the fiber bundle reduces to an $\aut(M_n,D_n)$-bundle. Here the group $\aut(M_n,D_n)$ consists of those automorphisms of $M_n$ which leave the diagonal subalgebra $D_n$ invariant. Since all automorphisms of $M_n$ are inner, elements of $\aut(M_n,D_n)$ are precisely those inner automorphisms implemented by a unitary that normalizes $D_n$. One can easily show that a unitary $u$ normalizes $D_n$, meaning that $uD_nu^*$ as well as $u^*D_nu$ are contained in $D_n$, if and only if $u$ is the product of a permutation matrix with a unitary diagonal matrix. We write $\calU(N_{D_n})$ for the group of such unitary normalizers. A rank $n$ vector bundle $\V$ over a compact space $X$ has an atlas with $\calU(N_{D_n})$-valued transition functions if and only if the principal bundle associated to the $C^*$-algebra of compact operators $\K(\Gamma(\V))$, which is an $n$-homogeneous $C^*$-algebra over $X$, reduces to a $\aut(M_n,D_n)$-bundle. Here $\Gamma(\V)$ denotes the space of continuous sections of $\V$. Note that by the classical Serre--Swan theorem, every finitely generated projective module over $C(X)$ is of the form $\Gamma(\V)$ for some vector bundle $\V$. Thus we see that the conditions in \cite{frausino2023} and \cite{LiRenault:2019} for the existence of topological graph models and Cartan subalgebras, respectively, are the same up to a unitary diagonal matrix. This unitary diagonal matrix encodes the twist, and accounts for the fact that \cite{frausino2023} only considered nontwisted topological graphs. 

We connect the findings of \cite{frausino2023} and \cite{LiRenault:2019}. First of all we define twisted topological correspondences, whose definition in this generality seems to be missing in the literature. We prove that they can be described by principal $\calU(N_{D_n})$-bundles. We then show that a sufficiently nice $C^*$-correspondence $\E$ between two commutative $C^*$-algebras can be constructed from a twisted topological correspondence if and only if the algebra of compact operators $\K(\E)$ has a Cartan subalgebra containing the image of the left action of $\E$, if and only if the vector bundle associated to $\E$ has an atlas whose transition functions take values in $\calU(N_{D_n})$, and which is compatible with the left action in a certain sense. In fact we prove something stronger, namely that there is a bijection between isomorphism classes of such objects, under suitable notions of isomorphism. 

Our results provide a different, more geometric point of view on topological graphs and correspondences. In particular, it allows one to approach these objects from the perspective of fiber and principal bundles. The latter are extremely well studied, with an abundance of deep methods and results. This theory can now be utilised to study topological graphs and correspondences.

On the other hand, the Cartan-subalgebra perspective is useful when considering another natural question: To what extent can one reconstruct a topological correspondence from its $C^*$-correspondence? In particular, this asks if isomorphism between the associated $C^*$-correspondences implies isomorphism of topological correspondences. This type of rigidity question has a long history, and has been considered in various contexts in operator algebras, where usually one tries to reconstruct a graph, group or dynamical system from the associated $C^*$-algebra. In \cite{frausino2023} the authors considered to what extend one can reconstruct a topological graph from its $C^*$-correspondence. Related questions were raised in \cite{davidsonRoydor:2011,davidsonKatsoulis:2011}, and in the context of multivariable dynamical systems \cite{KakariadisKatsoulis:2012,KakariadisKatsoulis:2014,Katsoulis:2018,KatsoulisRamsey:2022}. We build on the work in \cite{frausino2023} and try to reconstruct twisted topological correspondences from the associated $C^*$-correspondences. It is not clear exactly how much information about the topological correspondence is contained in the right Hilbert module structure. However, we show how in certain special cases one can reconstruct the topological correspondence as well as the twist from the left action of the $C^*$-correspondence. 

To see whether the results here can be generalized to the setting of groupoid correspondences, possibly using noncommutative bundle theory, is the subject of future research. 

\subsection{Summary of the paper}
We start in Section \ref{sect:twisted coverings and principal bundles} by reviewing the basic theory of principal bundles and fiber bundles. We explicitly describe some basic constructions in order to fix our notation. 

We then define \emph{twisted coverings} of a locally compact Hausdorff space $X$, which are coverings together with a Hermitian line bundle on the covering space. In Theorem \ref{thm:twisted coverings and principal bundles}, they are shown to be in bijection with principal $\calU(N_{D_n})$-bundles over $X$. Here $\calU(N_{D_n})$ is the group of unitary normalizers of the canonical diagonal $D_n$ inside of the $n\times n$-matrices $M_n$. The last result of Section \ref{sect:twisted coverings and principal bundles} shows that if a principal bundle and a twisted covering correspond to each other as in Theorem \ref{thm:twisted coverings and principal bundles}, then they induce isomorphic Hilbert $C_0(X)$-modules. 

Theorem \ref{thm:twisted coverings and principal bundles} is what allows us to study twisted topological correspondences through the lens of principal and fiber bundles. It is also the most important tool for the main result of Section \ref{sect:correspondences}. Take a sufficiently nice $C^*$-correspondence $\E$ between commutative $C^*$-algebras, and write $\V$ for the vector bundle obtained via Serre--Swan. Then Theorem \ref{thm:corr and atlas one to one} establishes a bijection between isomorphism classes of twisted topological correspondences associated to $\E$, and atlases of $\V$ whose transition functions take values in $\calU(N_{D_n})$, and which are compatible with the left action in a certain sense. This generalizes \cite[Corollary 6.2]{frausino2023} in two different directions, namely taking both twists as well as the left action into account. 

The strategy for proving Theorem \ref{thm:corr and atlas one to one} is to first associate an atlas to a topological correspondence (Section \ref{sect:atlas from corr}), then associate a topological correspondence to an atlas (Section \ref{sec:corr from atlas}), and show that the two constructions are inverse to each other. All that is then left to check in the proof of Theorem \ref{thm:corr and atlas one to one} is that the constructions respect isomorphism classes. 

The same strategy of proof, namely to establish both directions of the bijection separately and then verify that it respects isomorphism classes, is used in Section \ref{sect:perspective of the compacts}. There we prove in Theorems \ref{thm:cartan and atlas} and \ref{thm:corr to cartan} that similar bijections exist for conjugacy classes of Cartan subalgebras in the $C^*$-algebra of compact operators on $\E$. 

To illustrate these results further, let $\E$ be a (sufficiently nice) $C^*$-correspondence from $C_0(Y)$ to $C_0(X)$, with left action $\varphi:C_0(Y)\to\K(\E)$, and let $\V$ be the vector bundle such that $\E\cong\Gamma_0(\V)$ as right Hilbert $C_0(X)$-modules. Denote the collection of all twisted $Y$-$X$-correspondences associated to $\E$ by $\{(Z,r,s,\calL)\}$. More precisely, this means that $Z$ is a locally compact Hausdorff space, $r:Z\to Y$ is a continuous map, $s:Z\to X$ is a local homeomorphism, and $\calL$ is a Hermitian line bundle on $Z$. Furthermore, $(Z,r,s,\calL)$ being associated to $\E$ means that $\E$ is isomorphic to the $C^*$-correspondence $\Gamma_0(Z,r,s,\calL)$ obtained from $(Z,r,s,\calL)$. 
Let $\{(U_i,h_i)\}$ be the collection of all atlases for $\V$ which are normalizing (Definition \ref{def:normalizing atlas}) and diagonalize the left action $\varphi$ (Definition \ref{def:diagonalizing atlas}). Similarly, let $\{D\subset\K(\E)\}$ be the collection of all Cartan subalgebras $D$ of $\K(\E)$ which contain the image of $\varphi$. Theorems \ref{thm:corr and atlas one to one}, \ref{thm:cartan and atlas} and \ref{thm:corr to cartan} can most succinctly be summarized in the diagram 
\begin{equation}\label{diag:triangle}\begin{tikzcd}
	& {\{(U_i,h_i)\}} \\
	{\{(Z,r,s,\mathcal{L})\}} && {\{D\subset\mathcal{K}(\mathcal{E})\}}
	\arrow["{\text{Thm }\ref{thm:cartan and atlas}}", tail reversed, from=1-2, to=2-3]
	\arrow["{\text{Thm }\ref{thm:corr and atlas one to one}}", tail reversed, from=2-1, to=1-2]
	\arrow["{\text{Thm }\ref{thm:corr to cartan}}"', tail reversed, from=2-1, to=2-3]
\end{tikzcd}\end{equation}
where all the arrows represent bijections between isomorphism classes of the objects. As shown in Proposition \ref{prop:commutativity of diagram}, the diagram commutes up to isomorphism. 

The last result of Section \ref{sect:perspective of the compacts} is Proposition \ref{prop:isos for different Cstar correspondences}. There we do not fix a $C^*$-correspondence. Instead we consider atlases, Cartan subalgebras or topological correspondences associated to two different $C^*$-correspondences, and give conditions for when they are isomorphic. 

Section \ref{sec:applications} concerns applications and examples of Diagram (\ref{diag:triangle}). We first consider twisted topological graphs over spheres, and are able to reproduce results from \cite{LiRenault:2019} with different methods. In Section \ref{sect:covering with nontrivial vector bundle} we give an example of a covering with nontrivial associated vector bundle, which answers the question raised in \cite{frausino2023} of whether such coverings exist. We then study to what extent a twisted topological correspondence can be reconstructed from its associated $C^*$-correspondence. We show that in particular examples this is possible using the left action of the $C^*$-correspondence. The main tools are the results of Section \ref{sect:perspective of the compacts}. We study the abelian subalgebra $C^*(\im\varphi, C_0(X))$ of $\K(\E)$, which is generated by $C_0(X)$ regarded as a subset of $\K(\E)$, and by the image of the left action $\varphi$. We also consider its double commutant $C^*(\im\varphi, C_0(X))''$ inside $\K(\E)$. Both are abelian subalgebras of $\K(\E)$, which are contained in any Cartan subalgebra of $\K(\E)$ containing the image of $\varphi$. We study how much information these subalgebras contain in general. In \cite{frausino2023}, the authors ask what data one needs in addition to local conjugacy of topological graphs to obtain an isomorphism between the associated $C^*$-correspondences. In Corollary \ref{cor:iso and local conjugacy}, we answer this question in the special case that the fibers of either $C^*(\im\varphi, C_0(X))$ or $C^*(\im\varphi, C_0(X))''$ have constant rank over $X$.

\textit{Acknowledgements:} I would like to thank Tristan Bice, Kevin Brix, Nina Henn, Georg Jakob, Alistair Miller, and Karen Strung for helpful conversations. I would also like to thank Karen Strung for carefully reading parts of the manuscript.   

\section{Twisted covering maps and their associated principal bundles} \label{sect:twisted coverings and principal bundles}

\subsection{Principal bundles and fiber bundles}\label{sct:principal bundles and fiber bundles}

The main reference for this section is \cite{Husemoller:1993}. A \emph{bundle} is a triple $(Z,p,X)$ consisting of two topological spaces $Z$ and $X$ called the \emph{total space} and the \emph{base space}, respectively, and a continuous map $p$ from $Z$ to $X$ called the \emph{bundle projection}. Note that $p$ is not required to be surjective. The \emph{fiber over $x\in X$} is the space $p^{-1}(x)$. If $(Z,p,X)$ and $(Z',p',X)$ are two bundles over the same base space $X$, then a map $\phi$ from $Z$ to $Z'$ is called \emph{base-preserving} if $p'\circ\phi=p$. A continuous base-preserving map is called a \emph{bundle morphism}.

Let $G$ be a topological group. Let $P$ be a locally compact Hausdorff space with a free right $G$-action that admits a continuous translation function. We write $\pi:P\to P/G$ for the quotient map. A bundle of the form $(P,\pi,P/G)$ is called a \emph{principal $G$-bundle} (see \cite[Chapter 4, Definition 2.2]{Husemoller:1993}). Every fiber of a principal $G$-bundle is homeomorphic to $G$. A bundle morphism between two principal $G$-bundles is called a \emph{principal bundle morphism} provided it intertwines the $G$-actions. 

Let $\mathcal{P}=(P,p,X)$ be a principal $G$-bundle. Let $F$ be a locally compact Hausdorff space on which $G$ acts from the left. Take the quotient of the product space $P\times F$ by the equivalence relation that identifies $(\alpha,t)$ with $(\alpha\cdot g,g^{-1}\cdot t)$ for all $g\in G$. This is the total space of a bundle whose bundle projection is given by $p_F([\alpha,t])=p(\alpha)$. A bundle arising in this way is called a \emph{fiber bundle associated to $\mathcal{P}$} with fiber $F$, and we will denote it by $\mathcal{P}\times F$. The group $G$ is called the \emph{structure group} of the fiber bundle. 

Let $\mathcal{P}\times F$ and $\mathcal{P}'\times F$ be two fiber bundles with underlying principal $G$-bundles $\mathcal{P}=(P,p,X)$ and $\mathcal{P}'=(P',p',X)$. Let $\phi:P\to P'$ be a principal bundle morphism. The induced map $\phi\times 1_F$ from $P\times F$ to $P'\times F$ descends to the quotient, and thus to a bundle morphism from $\mathcal{P}\times F$ to $\mathcal{P}'\times F$. A bundle morphism of this form is called a \emph{fiber bundle morphism}. 

Let $(Z,p,X)$ be a fiber bundle with fiber $F$ and structure group $G$. A \emph{chart} is a pair $(U,h)$ consisting of an open set $U$ in $X$ and a fiber bundle isomorphism from the trivial fiber bundle $U\times F$ to $p^{-1}(U)$. An \emph{atlas} is a collection of charts $(U_i,h_i)_{i\in I}$ such that the $U_i$ cover $X$. A fiber bundle that admits an atlas is called \emph{locally trivial}. We always demand fiber and principal bundles to be locally trivial. Set $g_{ij}(x)\coloneqq h_i^{-1}\circ h_j(x,\cdot)$, which is a map on $F$. Each $g_{ij}(x)$ is given by an element of $G$, and the map $x\mapsto g_{ij}(x)$ from $U_i\cap U_j$ to $G$ is continuous. One can show that the collection of functions $g_{ij}$ satisfies the cocycle identity $g_{ij}g_{jk}=g_{ik}$ for all $i,j,k\in I$ such that $U_i\cap U_j\cap U_k\neq 0$. 

In general, if $G$ is a topological group, then a collection of $G$-valued continuous functions $g_{ij}$ from $U_i\cap U_j$ such that the cocycle identity holds is called a \emph{system of transition functions}. We have just seen that given a fiber bundle, one obtains a system of transition functions. On the other hand, if we are given a system of transition functions and an action of $G$ on a topological space $F$, then we can construct a fiber bundle $\mathcal{Z}$ over $X$ with fiber $F$ and structure group $G$. This fiber bundle then has an atlas whose transition functions are exactly the $g_{ij}$. For the construction of $\mathcal{Z}$ see \cite[Chapter 5,Theorem 3.2]{Husemoller:1993}. It roughly works in the following way: Let \[W\coloneqq \bigsqcup_{i\in I}U_i\times F, \]
and write elements of $W$ as triples $(x,f,i)$ with $x\in U_i$, $f\in F$ and $i\in I$. Define an equivalence relation which identifies $(x,f,i)$ and $(y,f',j)$ if $x=y$ and $f'=g_{ji}f$. The total space of $\mathcal{Z}$ is the quotient of $W$ by this equivalence relation. The bundle projection of $\mathcal{Z}$ maps $[x,f,i]$ to $x$.

A (complex) \emph{vector bundle} is a bundle all of whose fibers are finite dimensional complex vector spaces, and which is locally trivial as described above. Note that a fiber bundle whose fiber is a finite dimensional complex vector space, is in particular a vector bundle. However, vector spaces do not need to have the same vector space in each fiber. The \emph{rank} of a vector bundle over a point in the base space is the dimension of the fiber over that point. We even allow the rank to be zero. This is the case if the bundle projection is not surjective, though we demand its image be open. A morphism of vector bundles is a bundle morphism that restricts to a linear map in each fiber.

 \subsection{Constructing morphisms of principal bundles}
 Let $G$ and $H$ be topological groups and $X$ be a topological space.
Take a principal $G$-bundle $\mathcal{P}=(P,p,X)$ and a continuous group homomorphism $\rho$ from $G$ to $H$. Choose an atlas $(U_i,h_i)_{i\in I}$ of $\mathcal{P}$ and write $\{g_{ij}\}$ for its transition functions. Defining $h_{ij}\coloneqq\rho(g_{ij})$ yields a new system of transition functions. There is a unique principal $H$-bundle $\mathcal{Q}=(Q,q,X)$ associated to these new transition functions \cite[Chapter 5,Theorem 3.2]{Husemoller:1993}.

\begin{lemma}\label{lem:inducing principal bundle morphism}
    The map $\rho$ induces a principal bundle morphism from $\mathcal{P}$ to $\mathcal{Q}$. If $\rho$ is surjective, then this morphism is as well.
\end{lemma}
\begin{proof}
Recall that the total space $P$ of $\mathcal{P}$ is a quotient of the space $W_G=\bigsqcup_{i\in I}U_i\times G$, with respect to the equivalence relation that identifies $(x,g,i)$ and $(y,h,j)$ if and only if $x=y$ and $h=g_{ji}g$. The same description with $W_H$ instead of $W_G$ and $\rho(g_{ij})$ instead of $g_{ij}$ holds for the total space $Q$ of $\mathcal{Q}$. Define a map $\phi$ from $W_G$ to $W_H$ by sending $(x,g,i)$ to $(x,\rho(g),i)$. If $(x,g,i)$ and $(x,h,j)$ are equivalent then we have
\[\phi(x,g,i)=(x,\rho(g),i)=(x,\rho(g_{ij}h),i)=(x,\rho(g_{ij})\rho(h),i),\]
and $(x,\rho(g_{ij})\rho(h),i)$ is equivalent to $(x,\rho(h),j)=\phi(x,h,j)$. Hence $\phi$ respects the equivalence relation, and thus induces a map $\Phi$ from $P$ to $Q$. Since $\phi$ is continuous, $\Phi$ is as well. It respects the bundle projections $p$ and $q$, and hence defines a bundle morphism from $\mathcal{P}$ to $\mathcal{Q}$. To show that $\Phi$ is a principal bundle morphism we need to show that it is equivariant with respect to the $G$-action on $P$ and the $H$-action on $Q$, namely that $\Phi(\alpha\cdot t)=\Phi(\alpha)\cdot\rho(t)$ for any $\alpha\in P$ and $t\in G$. For every $(x,g,i)\in W_G$ and any $t\in G$ we have
\[\Phi([x,g,i]\cdot t)=\Phi([x,gt,i])=[x,\rho(gt),i]=[x,\rho(g),i]\rho(t)=\Phi([x,g,i])\cdot\rho(t).\]

If $\rho$ is surjective, then it immediately follows from the above construction that $\Phi$ is surjective as well. 
\end{proof}

\begin{lemma}\label{lem:inducing fiber bundle morphism}
    Let $F_G$ and $F_H$ be topological spaces that $G$ and $H$ act on, respectively. Let $\psi$ be a continuous map from $F_G$ to $F_H$ that is equivariant with respect to $\rho$ and the actions of $G$ and $H$. Let $\mathcal{M}$ and $\mathcal{N}$ be fiber bundles associated to $\mathcal{P}$ and $\mathcal{Q}$ with fibers $F_G$ and $F_H$, respectively. Sending $[\alpha,y]$ to $[\Phi(\alpha),\psi(y)]$ defines a morphism $\Psi:\mathcal{M}\to\mathcal{N}$ of fiber bundles. If $\rho$ and $\psi$ are surjective, then $\Psi$ is as well.
\end{lemma}
\begin{proof}
    The proof is essentially the same as the proof of Lemma \ref{lem:inducing principal bundle morphism}.
\end{proof}

\subsection{Twisted covering maps} \label{sect:twisted coverings}

Let $X$ and $Z$ be locally compact Hausdorff spaces. We call a continuous map $s:Z\to X$ a \emph{covering map} if the following holds: For every point $x\in s(Z)$ there exists an open neighborhood $U$ of $x$ contained in $s(Z)$, a discrete space $D_x$ and a homeomorphism $\phi$ from $U\times D_x$ to $s^{-1}(U)$ such that $s\circ\phi=p_1$. Here $p_1$ is the map from $U\times D_x$ to $U$ mapping $(x,d)$ to $x$ for every $d\in D_x$. A covering map is called \emph{$n$-sheeted} if $s$ is surjective and there exists $n\in\N$ such that for all $x\in X$, the space $D_x$ has exactly $n$ elements.

A map $s$ between locally compact Hausdorff spaces $Z$ and $X$ is called a \emph{local homeomorphism} if every $z\in Z$ has a neighborhood $U$ whose image $s(U)$ is open, and such that $s$ restricts to a homeomorphism from $U$ to $s(U)$. If $s$ is proper onto its image, then it is a covering map. Conversely, every covering map is a local homeomorphism.

A \emph{metric} on a vector bundle $\V=(V,p,X)$ is a continuous map $\beta$ from the total space of $\V\oplus\V$ to $\C$ such that for each $x\in X$, restricting $\beta$ to $p^{-1}(x)\times p^{-1}(x)$ yields an inner product on $p^{-1}(x)$. By \cite[Theorem 2.5]{partI}, if there exists a metric on $\V$ then it is unique up to isometry (note that \cite[Theorem 2.5]{partI} states paracompactness of $X$ as an assumption, but this is only used to ensure the existence of a metric). A \emph{Hermitian line bundle} $\calL$ over $X$ is a complex line bundle together with a metric $\beta$. By \cite[Chapter 5, Theorem 7,4]{Husemoller:1993}, a Hermitian line bundle has an atlas whose transition functions take values in $\T$. The converse is true as well, that is, a complex line bundle whose structure group reduces to $\T$ has an Hermitian metric. 
 \begin{definition}\label{def:twisted local homeo}
     Let $X$ be a locally compact Hausdorff space. A \emph{twisted local homeomorphism over $X$} is a triple $(Z,s,\calL)$ consisting of a local homeomorphism $s:Z\to X$, for $Z$ a locally compact Hausdorff space, and a Hermitian line bundle $\calL$ with base space $Z$. If $s$ is a covering map, then we call $(Z,s,\calL)$ a \emph{twisted covering}.
     
      Two twisted local homeomorphisms $(Z,s,\calL)$ and $(Z',s',\calL')$ of $X$ are called isomorphic if there exists a homeomorphism $\phi$ from $Z$ to $Z'$ such that $s'\circ\phi=s$, and such that the pullback bundle $\phi^*\calL'$ is isomorphic to $\calL$. In particular, two local homeomorphisms or covering maps $s:Z\to X$ and $s':Z'\to X$ are isomorphic if there exists a homeomorphism $\phi$ from $Z$ to $Z'$ such that $s'\circ\phi=s$. 
 \end{definition}

The permutation group $S_n$ acts on the set with $n$ elements $\{1,\dots,n\}$ in the obvious way. Let $\mathcal{Z}=(Z,p,X)$ be a fiber bundle with fiber $\{1,\dots,n\}$ and structure group $S_n$ over a locally compact Hausdorff space $X$. It follows from local triviality of $\mathcal{Z}$ that $p$ is an $n$-sheeted covering map. We show in Appendix \ref{appendix A} that every $n$-sheeted covering map over $X$ arises this way, and that two such fiber bundles are isomorphic if and only if the associated coverings are isomorphic. Thus instead of studying the twisted $n$-sheeted covering $(Z,s,\calL)$, we can study the pair $(\mathcal{Z},\calL)$ of fiber bundles, where $\mathcal{Z}=(Z,p,X)$ has structure group $S_n$ and fiber $\{1,\dots,n\}$. 

\begin{definition}\label{def:locally trivializable pair}
    Let $\calL=(L,p_L,Z)$ and $\mathcal{Z}=(Z,p_Z,X)$ be fiber bundles. We call the pair $(\mathcal{Z},\calL)$ \emph{simultaneously trivializable} if there exists an open cover $\{U_i\}_{i\in I}$ of $X$ such that $\mathcal{Z}$ is trivial over each $U_i$, and $\calL$ is trivial over each $p_Z^{-1}(U_i)$. In other words, there exist atlases $(U_i,h_i)_{i\in I}$ and $(p_Z^{-1}(U_i),k_i)_{i\in I}$ of $\mathcal{Z}$ and $\calL$, respectively.
\end{definition}

 \begin{lemma}\label{lem:simultaneously trivializable}
     Let $(\mathcal{Z},\calL)$ be a pair consisting of a fiber bundle $\mathcal{Z}=(Z,p_Z,X)$ with fiber $\{1,\dots,n\}$ and structure group $S_n$, and a fiber bundle $\calL$ whose base space is $Z$. Then $(\mathcal{Z},\calL)$ is simultaneously trivializable. 
 \end{lemma}
 \begin{proof}
Take any point $x\in X$ and write $p_Z^{-1}(x)=\{y_1,\dots,y_n\}$. We can find disjoint open sets $U_1,\dots,U_n$ in $Z$ such that $y_i$ lies in $U_i$ and $\calL$ is trivial over $U_i$, for all $i=1,\dots,n$. 

Let $U$ be an open neighborhood of $x$ over which $\mathcal{Z}$ is trivial. Then $V\coloneqq p_Z(U_1)\cap \dots\cap p_Z(U_n)\cap U$ is an open neighborhood of $x$ such that $\mathcal{Z}$ is trivial over $V$, and $\calL$ is trivial over $p_Z^{-1}(V)$. This shows that $(\mathcal{Z},\calL)$ is simultaneously trivializable.
 \end{proof}

\subsection{Principal $\T^n\rtimes S_n$-bundles and twisted coverings}\label{app:simultaneously trivializable pairs}
The permutation group $S_n$ acts on $\T^n$ by associating to $\sigma\in S_n$ the map that sends $(\lambda_1,\dots,\lambda_n)\in\T^n$ to $(\lambda_{\sigma(1)},\dots,\lambda_{\sigma(n)})$. The semidirect product $\T^n\rtimes S_n$ with respect to this action has multiplication and inverse given by 
\[(\lambda,\sigma_1)(\mu,\sigma_2)=(\lambda(\sigma_1\cdot\mu),\sigma_1\sigma_2)\quad\text{and}\quad (\lambda,\sigma)^{-1}=(\sigma^{-1}\cdot\lambda^{-1},\sigma^{-1}),\]
for $\sigma,\sigma_1,\sigma_2\in S_n$ and $\lambda,\mu\in \T^n$. 

Write $I_n\coloneqq\{1,\dots,n\}$. Then $\T^n\rtimes S_n$ acts on the Cartesian product $I_n\times \C$ via
\begin{equation}\label{eq:action of semidirect product}
    (\lambda,\sigma)\cdot(i,z)=(\sigma(i),\lambda_{\sigma(i)}z),
\end{equation}
for $\lambda\in \T^n$, $\sigma\in S_n$, $z\in \C$ and $i\in I_n$. 

 Let $(\mathcal{Z},\calL)$ be a pair consisting of a fiber bundle $\mathcal{Z}=(Z,s,X)$ with fiber $I_n$ and structure group $S_n$, and a Hermitian line bundle $\calL=(L,p_L,Z)$. The the composition $s\circ p_L$ defines a bundle $\mathcal{M}=(L,s\circ p_L,X)$. 
 \begin{lemma}\label{lem:fiber bundle from pair}
     The bundle $\mathcal{M}$ is a fiber bundle with structure group $\T^n\rtimes S_n$ and fiber $\C\times I_n$. The action of the structure group on the fiber is given by (\ref{eq:action of semidirect product}).
 \end{lemma}
 \begin{proof}
    By Lemma \ref{lem:simultaneously trivializable} the pair $(\mathcal{Z},\calL)$ is simultaneously trivializable. Thus there exist atlases $(U_i,h_i)_{i\in I}$ and $(s^{-1}(U_i),k_i)_{i\in I}$ of $\mathcal{Z}$ and $\calL$, respectively. We obtain the following commutative diagram: 
\begin{equation}\label{diagram fiber bundle from pair}\begin{tikzcd}
	{U_i\times I_n\times \C} & {s^{-1}(U_i)\times \C} & {p_L^{-1}(s^{-1}(U_i))} \\
	{U_i\times I_n} & {s^{-1}(U_i)} \\
	{U_i}
	\arrow[from=1-1, to=2-1]
	\arrow[from=2-1, to=3-1]
	\arrow[from=1-2, to=2-2]
	\arrow["{h_i}", from=2-1, to=2-2]
	\arrow["{h_i\times \mathrm{id}_{\C}}", from=1-1, to=1-2]
	\arrow["{k_i}", from=1-2, to=1-3]
	\arrow["{s}", from=2-2, to=3-1]
	\arrow["{p_L}", from=1-3, to=2-2]
\end{tikzcd}\end{equation}
 Write $\phi_i\coloneqq k_i\circ(h_i\times\id_{\C})$. It then follows from the above diagram that $\{U_i,\phi_i\}_{i\in I}$ is an atlas of $\mathcal{M}$. In particular $\mathcal{M}$ is locally trivial and has fiber $I_n\times \C$. 
 
 Let $\{\sigma_{ij}:U_i\cap U_j\to S_n\}$ and $\{t_{ij}:s^{-1}(U_i)\cap s^{-1}(U_j)\to\T\}$ be the transition functions of $\mathcal{Z}$ and $\calL$ associated to the atlases above, respectively. We write $\sigma^x_{ij}$ for $\sigma_{ij}(x)$. Take $(x,k,z)$ in $U_i\cap U_j\times I_n\times \C$.  We then have 
 \begin{align*}
 \phi_i^{-1}\circ\phi_j(x,k,z)&=(h_i^{-1}\times\id_{\C})\circ k_i^{-1}\circ k_j\circ (h_j\times\id_{\C})(x,k,z)\\&=(h_i^{-1}\times\id_{\C})\circ k_i^{-1}\circ k_j(h_j(x,k),z)\\&=(h_i^{-1}\times\id_{\C})(h_j(x,k),t_{ij}(h_j(x,k))z)\\&=(x,\sigma_{ij}^x(k),t_{ij}(h_j(x,k))z)\\&=(x,\sigma_{ij}^x(k),t_{ij}(h_i(x,\sigma_{ij}^x(k)))z).
 \end{align*}
 Let $\{g_{ij}:U_i\cap U_j\to \T^n\rtimes S_n\}$ denote the transition functions of the atlas $\{U_i,\phi_i\}$. The above calculation implies 
 \begin{equation}\label{eq:trans fcts I}g_{ij}(x)(k,z)=(\sigma_{ij}^x(k),t_{ij}(h_i(x,\sigma_{ij}^x(k)))z)
 \end{equation}
 for all $x\in U_i\cap U_j$ and all $(k,z)\in I_n\times \C$. Using Equation (\ref{eq:action of semidirect product}), we obtain that for any $x\in U_i\cap U_j$, $g_{ij}(x)$ is the element of $\T^n\rtimes S_n$ given by $(t_{ij}(h_i(x,\cdot)),\sigma_{ij}^x)$, with $t_{ij}(h_i(x,\cdot))$ in $\T^n\cong\prod_{I_n}\T$ mapping $k\in I_n$ to $t_{ij}(h_i(x,k))$. Section \ref{sct:principal bundles and fiber bundles} now tells us that $\mathcal{M}$ is a fiber bundle with fiber $I_n\times \C$ and structure group $\T^n\rtimes S_n$.
 \end{proof}
\begin{lemma}\label{lem:pair from fiber bundle}
    Let $\mathcal{P}=(P,p,X)$ be a principal $\T^n\rtimes S_n$-bundle. Then there are a Hermitian line bundle $\calL=(L,p_L,Z)$, and a fiber bundle $\mathcal{Z}=(Z,s,X)$ with structure group $S_n$ and fiber $\{1,\dots,n\}$, such that the pair $(\mathcal{Z},\calL)$ is simultaneously trivializable. Furthermore, the bundle $(L,s\circ p_L,X)$ is a fiber bundle associated to $\mathcal{P}$.
\end{lemma}
\begin{proof}
     Let $\mathcal{M}=(L,p_M,X)$ be the fiber bundle with fiber $I_n\times \C$ associated to $\mathcal{P}$. Here we use the action of $\T^n\rtimes S_n$ on $I_n\times \C$ as given by (\ref{eq:action of semidirect product}). The projection from $\T^n\rtimes S_n$ onto $S_n$ is a group homomorphism. The projection from $I_n\times \C$ to $I_n$ is equivariant. From Lemma \ref{lem:inducing fiber bundle morphism} we obtain a fiber bundle $\mathcal{Z}=(Z,s,X)$ and a surjective fiber bundle morphism $p_L$ from $\mathcal{M}$ to $\mathcal{Z}$. Recall that $\mathcal{Z}$ was constructed from an atlas $(U_i,\phi_i)$ of $\mathcal{M}$. By construction we have $p_M=s\circ p_L$ and there is an atlas $(U_i,h_i)$ of $\mathcal{Z}$ such that the diagram
\[\begin{tikzcd}
	{U_i\times I_n\times \C} & {p_L^{-1}(s^{-1}(U_i))} \\
	{U_i\times I_n} & {s^{-1}(U_i)}
	\arrow[from=1-1, to=2-1]
	\arrow["{p_L}", from=1-2, to=2-2]
	\arrow["{h_i}", from=2-1, to=2-2]
	\arrow["{\phi_i}", from=1-1, to=1-2]
\end{tikzcd}\]
 commutes, where the arrow on the left hand side is the projection onto the first two factors.
 Now define a map $k_i$ from $s^{-1}(U_i)\times \C$ to $p_L^{-1}(s^{-1}(U_i))$ by $k_i\coloneqq\phi_i\circ(h_i^{-1}\times\id_{\C})$. We obtain the commutative diagram
\begin{equation}\label{diagram pair from fiber bundle}\begin{tikzcd}
	{U_i\times I_n\times \C} & {s^{-1}(U_i)\times \C} & {p_L^{-1}(s^{-1}(U_i))} \\
	{U_i\times I_n} & {s^{-1}(U_i)} \\
	{U_i}
	\arrow["{h_i\times\mathrm{id}_{\C}}", from=1-1, to=1-2]
	\arrow["{k_i}", dashed, from=1-2, to=1-3]
	\arrow["{\phi_i}", curve={height=-24pt}, from=1-1, to=1-3]
	\arrow["{h_i}", from=2-1, to=2-2]
	\arrow[from=1-1, to=2-1]
	\arrow[from=1-2, to=2-2]
	\arrow["{p_L}", from=1-3, to=2-2]
	\arrow["{s}", from=2-2, to=3-1]
	\arrow[from=2-1, to=3-1]
\end{tikzcd}\end{equation}
This shows that $p_L:L\to Z$ defines a bundle with fiber $\C$, and that $(p^{-1}(U_i),k_i)$ is an atlas for this bundle. Write $\calL=(L,p_L,Z)$. 

We now calculate the transition functions of $\calL$. First note that the transition functions $g_{ij}$ of $\mathcal{M}$ have the form 
\begin{equation}\label{eq:trans fct II}
g_{ij}(x)=(t_{ij}(x,1),\dots,t_{ij}(x,n),\sigma_{ij}^x)\in \T^n\rtimes S_n,\qquad x\in U_i\cap U_j,
\end{equation}
where $t_{ij}:U_i\cap U_j\times I_n\to\T$ is a continuous function, and $\sigma_{ij}:U_i\cap U_j\to S_n$ are the transition functions of $\mathcal{Z}$. For $(x,k,z)\in U_i\times I_n\times \C$ we calculate
\begin{align*}k_i^{-1}\circ k_j(h_j(x,k),z) &=k_i^{-1}\circ\phi_j(x,k,z)=h_i\times\id_{\C}(x,g_{ij}(x)(k,z))\\&=h_i\times\id_{\C}(x,\sigma_{ij}^x(k),t_{ij}(x,k)z)=(h_j(x,k),t_{ij}(x,k)z).
 \end{align*}
This shows that the transition functions of $\calL$ take values in $\T$. In conclusion $\calL$ is a fiber bundle with fiber $\C$, structure group $\T$ and an atlas $(p^{-1}(U_i),k_i)$. In particular the pair ($\mathcal{Z}$,$\calL$) is simultaneously trivializable.

Given the pair $(\mathcal{Z},\calL)$, we can form the bundle $(L,s\circ p_L,X)$. By Lemma \ref{lem:fiber bundle from pair} we obtain that it is a fiber bundle with fiber $I_n\times \C$ and structure group $\T^n\rtimes S_n$. It is clear from the fact that $p_M=s\circ p_L$, and from the diagrams (\ref{diagram fiber bundle from pair}) and (\ref{diagram pair from fiber bundle}), that the principal $\T^n\rtimes S_n$-bundle associated to $(L,s\circ p_L,X)$ is isomorphic to $\mathcal{P}$. Hence both constructions are inverse to each other. 
\end{proof}
Two locally trivializable pairs $(\mathcal{Z},\calL)$ and $(\mathcal{Z}',\calL')$ are called isomorphic if there exists an isomorphism of fiber bundles $\Phi$ from $\mathcal{Z}$ to $\mathcal{Z}'$ such that the pullback bundle $\Phi^*\calL'$ is isomorphic to $\calL$. More explicitly, the isomorphism $\Phi$ will be a homeomorphism between the total spaces $Z$ and $Z'$ that is compatible with the fiber bundle structure. Since the base space of $\calL'$ is $Z'$, we can use $\Phi$ to pull $\calL'$ back to $Z$. We then require that this pullback is isomorphic to $\calL$. 
\begin{proposition}\label{prop:simultaneously trivializable pairs and principal bundles}
The constructions from Lemmas \ref{lem:fiber bundle from pair} and \ref{lem:pair from fiber bundle} set up a bijection between isomorphism classes of principal $\T^n\rtimes S_n$-bundles and isomorphism classes of simultaneously trivializable pairs $(\mathcal{Z},\calL)$. 
\end{proposition} 
\begin{proof}
With Lemmas \ref{lem:fiber bundle from pair} and \ref{lem:pair from fiber bundle} in mind, the only thing left to show for the bijection is that two principal bundles $\mathcal{P}$ and $\mathcal{P}'$ are isomorphic if and only if the associated pairs $(\mathcal{Z},\calL)$ and $(\mathcal{Z}',\calL')$ are isomorphic. 

Assume that $\mathcal{P}$ and $\mathcal{P}'$ are isomorphic. Then the bundles $\mathcal{M}=\mathcal{P}\times (I_n\times \C)$ and $\mathcal{M}'=\mathcal{P}'\times(I_n\times \C)$ from the proof of Lemma \ref{lem:pair from fiber bundle} are isomorphic. This isomorphism descends to an isomorphism between $\mathcal{Z}$ and $\mathcal{Z}'$. We have the commutative diagram

\[\begin{tikzcd}
	L & Z \\
	&& X \\
	{L'} & {Z'}
	\arrow["{p_L}", from=1-1, to=1-2]
	\arrow["\cong"', from=1-1, to=3-1]
	\arrow["{s}", from=1-2, to=2-3]
	\arrow["\cong"', from=1-2, to=3-2]
	\arrow["{p_{L'}}", from=3-1, to=3-2]
	\arrow["{p_{Z'}}"', from=3-2, to=2-3]
\end{tikzcd}\]
which implies isomorphism of the pairs $(\mathcal{Z},\mathcal{L})$ and $(\mathcal{Z}',\mathcal{L}')$. 

On the other hand, assume that $(\mathcal{Z},\mathcal{L})$ and $(\mathcal{Z},\calL)$ are isomorphic. Then the fiber bundles $\mathcal{M}=(L,s\circ p_L,X)$ and $\mathcal{M}'=(L',s'\circ p_L',X)$ as in Lemma \ref{lem:fiber bundle from pair} are isomorphic. Recall that $\mathcal{P}$ and $\mathcal{P}'$ are defined as the principal bundles associated to $\mathcal{M}$ and $\mathcal{M}'$, respectively. Thus $\mathcal{P}$ and $\mathcal{P}'$ are isomorphic by \cite[Theorem 2.7, Chapter 5]{Husemoller:1993}. 
\end{proof}

 \begin{theorem} \label{thm:twisted coverings and principal bundles}
 There is a bijection between isomorphism classes of twisted $n$-sheeted coverings of $X$ and isomorphism classes of principal $\T^n\rtimes S_n$-bundles over $X$. The twist is trivial if and only if the principal $\T^n\rtimes S_n$-bundle reduces to an $S_n$-bundle.
\end{theorem}
\begin{proof}
By Proposition \ref{prop:covering spaces and principal bundles} and Definition \ref{def:twisted local homeo}, isomorphism classes of twisted $n$-sheeted coverings $(Z,s,\calL)$ of $X$ are in bijection with isomorphism classes of pairs $(\mathcal{Z},\calL)$, where $\mathcal{Z}=(Z,s,X)$ is a fiber bundle over $X$ with structure group $S_n$ and fiber $\{1,\dots,n\}$, and $\calL$ is a complex line bundle. Such pairs are simultaneously trivializable by Lemma \ref{lem:simultaneously trivializable}, and their isomorphism classes are in bijection with isomorphism classes of principal $\T^n\rtimes S_n$-bundles over $X$ by Proposition \ref{prop:simultaneously trivializable pairs and principal bundles}. This establishes the first claim. 

It follows from the Equations (\ref{eq:trans fcts I}) and(\ref{eq:trans fct II}) that we can choose the transition functions of the principal $\T^n\rtimes S_n$-bundle $\mathcal{P}$ to lie in $S_n$ if and only if we can choose the transition functions of $\calL$ to be trivial. Thus $\mathcal{P}$ reduces to an $S_n$-bundle if and only if $\calL$ is trivial. This concludes the proof of the theorem. 
\end{proof}

\begin{remark}\label{rem:different groups}
    Whenever we write $\C^n$ for the $n$-dimensional complex vector space, we have chosen an orthonormal basis $\{e_k\}_{k=1}^n$ which we refer to as the canonical or standard basis. Let $M_n$ be the $n\times n$ matrices with complex entries. Let $D_n$ be the diagonal subalgebra in $M_n$ obtained from the canonical basis. Write $N_{D_n}$ for the normalizers of $D_n$ in $M_n$, namely the set of all $a\in M_n$ such that $aD_na^*$ and $a^*D_na$ are contained in $D_n$. Let
    \[\calU(N_{D_n})=\{u\in U(n):u^*D_nu\subset D_n\text{  and } uD_nu^*\subset D_n\}\]
    be the unitary group of $N_{D_n}$. Any element of $\calU(N_{D_n})$ is of the form $\lambda\sigma$, where $\sigma$ is a permutation matrix, and $\lambda$ is a unitary diagonal matrix. We regard $\sigma$ as an element of $S_n$, in such a way that the canonical basis vector $e_k$ is sent to $e_{\sigma(k)}$ for every $k\in\{1,\dots,n\}$. We regard $\lambda$ as an element of $\T^n$ in the obvious way. Sending $\lambda\sigma$ to $(\lambda,\sigma)$ defines a group isomorphism from $\calU(N_{D_n})$ to $\T^n\rtimes S_n$. When working with Theorem \ref{thm:twisted coverings and principal bundles}, we will replace $\T^n\rtimes S_n$ by $\calU(N_{D_n})$. The reason is that there is an obvious action of $\calU(N_{D_n})$ on $\C^n$ by matrix multiplication. Therefore it seems more natural to consider $\calU(N_{D_n})$ when we want to associate vector bundles to $\mathcal{P}$. 
\end{remark}
\begin{remark}\label{rem:compatible atlases}
    Let $\mathcal{P}$ be a principal $\T^n\rtimes S_n$-bundle, and $(Z,s,\calL)$ the corresponding twisted covering of $X$ from Theorem \ref{thm:twisted coverings and principal bundles}. Write $\mathcal{Z}=(Z,s,X)$ for the fiber bundle with fiber $\{1,\dots,n\}$ and structure group $S_n$ associated to the covering $(Z,s)$. As noted in Remark \ref{rem:different groups} there is a canonical action of $\T^n\rtimes S_n\cong \calU(N_{D_n})$ on $\C^n$ by matrix multiplication. Write $\V=(V,p_V,X)$ for the vector bundle $\mathcal{P}\times\C^n$ associated to $\mathcal{P}$. Then an atlas of $\mathcal{P}$ over open sets $(U_i)_{i\in I}$ yields atlases $(U_i,h_i)_{i\in I}$, $(s^{-1}(U_i),k_i)_{i\in I}$ and $(U_i,\psi_i)_{i\in I}$ of $\mathcal{Z}$, $\mathcal{L}$ and $\mathcal{V}$, respectively. For $\mathcal{Z}$ and $\calL$, this follows from the constructions in the proofs of Lemmas \ref{lem:fiber bundle from pair} and \ref{lem:pair from fiber bundle}. For $\V$ it is true because $\V$ is a bundle associated to $\mathcal{P}$. If $\mathcal{M}=\mathcal{P}\times(\{1,\dots,n\}\times\C)$ is the bundle associated to $\mathcal{P}$ using the action (\ref{eq:action of semidirect product}), then setting $\phi_i\coloneqq k_i\circ(h_i\times\id_\C)$, yields an atlas $(U_i,\phi_i)_{i\in I}$ of $\mathcal{M}$. Write $g_{ij}$ for the transition functions of the atlas of $\mathcal{P}$. Then the $g_{ij}$ are also the transition functions for the atlases $(U_i,\psi_i)_{i\in I}$ and $(U_i,\phi_i)_{i\in I}$. The transition functions of the atlas $(U_i,h_i)$ are given by $\sigma_{ij}\coloneqq\rho\circ g_{ij}$, where $\rho:\T^n\rtimes S_n\to S_n$ is the quotient map. If $t_{ij}$ are the transition functions of the atlas $(s^{-1}(U_i),k_i)$, then the equality
    \[
    g_{ij}(x)=(t_{ij}(h_i(x,1)),\dots,t_{ij}(h_i(x,n)),\sigma_{ij}(x))\in \T^n\rtimes S_n
    \]
    holds for all $x\in U_i\cap U_j$. Whenever we are dealing with a twisted covering and the associated vector and $\T^n\rtimes S_n$-bundle, we will work with atlases chosen in this way. 
\end{remark}
\subsection{Isomorphism of Hilbert modules}
Recall from Section \ref{sect:twisted coverings} that a \emph{metric} on a vector bundle $\V=(V,p,X)$ is a continuous map $\beta$ from the total space of $\V\oplus\V$ to $\C$ such that for each $x\in X$, restricting $\beta$ to $p^{-1}(x)\times p^{-1}(x)$ yields an inner product on $p^{-1}(x)$. Given a metric $\beta$ on a vector bundle $\V$, we define the space $\Gamma_0(\V)$ of continuous sections of $\V$ that vanish at infinity. If $\xi$ and $\eta$ are elements of $\Gamma_0(\V)$, then we define an inner product by $\langle\xi,\eta\rangle(x)\coloneqq \beta(\xi(x),\eta(x))$. As shown in \cite[Proposition 3.1]{partI}, this turns $\Gamma_0(\V)$ into a right Hilbert $C_0(X)$-module. For the general theory of Hilbert modules see \cite{Lance:1995}. We usually write $\E$ for a Hilbert module, and $\langle\xi,\eta\rangle$ for the inner product of two elements $\xi,\eta\in\E$.

\begin{definition}\label{def:twisted local homeo module}
     Let $(Z,s,\calL)$ with $s:Z\to X$ be a twisted covering. Denote by $\langle\cdot,\cdot\rangle_\calL$ the $C_0(Z)$-valued inner product on $\Gamma_0(\calL)$ induced by the metric on $\calL$. Then, $\Gamma_0(Z,s,\calL)$ is defined to be the right Hilbert $C_0(X)$-module that as a vector space is $\Gamma_0(\calL)$, and with right multiplication and inner product given by
\begin{align*}
    \langle\xi,\eta\rangle(x)=\sum_{s(y)=x}\langle\xi,\eta\rangle_\calL(y)\qquad\text{and}\quad(\xi f)(y)=\xi(y)f(s(y)) 
\end{align*}
for $\xi,\eta\in \Gamma_0(\calL)$, $f\in C_0(X)$, $x\in X$, and $y\in Z$. 
\end{definition}
\begin{remark}\label{rem:metric atlas}
    Let $\V$ be a rank $n$ vector bundle. Assume that $\V$ has an atlas $(U_i,h_i)_{i\in I}$ whose transition functions take values in the unitary group $U(n)$. Then we obtain a metric $\beta$ on $\V$ by setting $\beta(v,w)\coloneqq \langle h_i^{-1}(v),h_i^{-1}(w)\rangle_{\C^n}$ for all $x\in X$ and $v,w\in\V_x$, where we omit the projection onto the second component of $U_i\times\C^n$. Here $\langle\cdot,\cdot\rangle_{\C^n}$ is the standard inner product on $\C^n$. 

    Conversely, by \cite[Chapter 5 Theorem 7,4]{Husemoller:1993}, a vector bundle with an hermitian metric has an atlas whose transition functions take values in $U(n)$.
\end{remark}
Let $\mathcal{P}$ be a principal $\calU(N_{D_n})$-bundle. As noted in Remark \ref{rem:different groups}, there is a canonical action of $\calU(N_{D_n})$ on $\C^n$ by matrix multiplication. The associated bundle $\mathcal{P}\times\C^n$ is a vector bundle, and it admits a metric since the structure group of $\mathcal{P}\times\C^n$ reduces to $\calU(N_{D_n})$. 
\begin{proposition}\label{prop:module iso}
    Let $(Z,s,\calL)$ be a twisted covering, and let $\mathcal{P}$ be the corresponding principal $\calU(N_{D_n})$-bundle from Theorem \ref{thm:twisted coverings and principal bundles}. Then 
    $\Gamma_0(\mathcal{P}\times\C^n)$ and $\Gamma_0(Z,s,\calL)$ are isomorphic as right Hilbert $C_0(X)$-modules. 
 
    If $(Z,s,\calL)$ and $(Z',s',\calL')$ are two isomorphic twisted coverings, then this induces an isomorphism between the corresponding principal $\calU(N_{D_n})$-bundles $\mathcal{P}$ and $\mathcal{P}'$ by Theorem \ref{thm:twisted coverings and principal bundles}, and the diagram
\begin{equation}\label{diag:module iso}\begin{tikzcd}
	{\Gamma_0(\mathcal{P}\times\mathbb{C}^n)} & {\Gamma_0(\mathcal{P}'\times\mathbb{C}^n)} \\
	{\Gamma_0(Z,s,\mathcal{L})} & {\Gamma_0(Z',s',\mathcal{L}')}
	\arrow["\cong", from=1-1, to=1-2]
	\arrow["\cong"', from=1-1, to=2-1]
	\arrow["\cong", from=1-2, to=2-2]
	\arrow["\cong", from=2-1, to=2-2]
\end{tikzcd}\end{equation}
commutes.
\end{proposition}
\begin{proof}

Write $\V=(V,p_V,X)$ for the associated vector bundle $\mathcal{P}\times\C$, and $\mathcal{Z}=(Z,s,X)$ for the fiber bundle corresponding to the covering $(Z,s)$. Let $(U_i,\psi_i)_{i\in I}$, $(U_i,h_i)_{i\in I}$ and  $(s^{-1}(U_i),k_i)_{i\in I}$ be atlases of $\V$, $\mathcal{Z}$ and $\calL$, respectively, chosen as in Remark \ref{rem:compatible atlases}. Then $(U_i,\phi_i)$ is an atlas of the associated bundle $\mathcal{P}\times(\{1,\dots,n\}\times\C)$, where $\phi_i\coloneqq k_i\circ(h_i\times\id_\C)$. 

Let $\xi:X\to V$ be a section of $\V$ vanishing at infinity. Write $\xi_i$ for the restriction of $\xi$ to $U_i$. Then $\psi_i^{-1}\circ\xi_i$ is a map from $U_i$ to $U_i\times\C^n$. Let $v(x)\in\C^n$ be such that $\psi_i^{-1}\circ\xi_i(x)=(x,v(x))$ for all $x\in U_i$. 

Define a map $\tau_i$ from $U_i\times\{1,\dots,n\}$ to $U_i\times\{1,\dots,n\}\times\C$ sending $(x,l)$ to $(x,l,v(x)_l)$, where $v(x)_l$ is the $l$-th component of $v(x)$. Then $t_i\coloneqq \phi_i\circ\tau_i\circ h_i^{-1}$ is a continuous map from $s^{-1}(U_i)$ to $p_L^{-1}(s^{-1}(U_i))$. We want to define a section $t:Z\to L$ of $\calL$ by setting $t(z)\coloneqq t_i(z)$ for $z\in s^{-1}(U_i)$. To show that this is well defined, we need to prove that if $z$ lies in the intersection of $s^{-1}(U_i)$ and $s^{-1}(U_j)$ for $i,j\in I$, then $t_i(z)=t_j(z)$. First, write $\psi_j^{-1}\circ\xi_j(x)=(x,w(x))$ for $w(x)\in\C^n$ and $x\in U_i\cap U_j$. Let $(x,l)\coloneqq h_i^{-1}(z)$ and $(x,k)\coloneqq h_j^{-1}(z)$. Clearly $\xi_i(x)=\xi_j(x)$. Then we have
\begin{align*}
    (x,v(x))=\psi_i^{-1}\circ\xi_i(x)=\psi_i^{-1}\circ\psi_j\circ\psi_j^{-1}\circ\xi_j(x)=(x,g_{ij}(x)w(x))
\end{align*}
and thus $v(x)=g_{ij}(x)w(x)$. Now calculate
\begin{align*}
    \phi_i^{-1}\circ t_j(z)&=\phi_i^{-1}\circ\phi_j\circ\tau_j\circ h_j^{-1}(z)=(x,g_{ij}\cdot(k,w(x)_k))\\&=(x,l,(g_{ij}(x)w(x))_l)=(x,l,v(x)_l)=\tau_i\circ h_i^{-1}(z).
\end{align*}
This implies $t_i(z)=\phi_i\circ\tau_i\circ h_i^{-1}(z)=t_j(z)$. Therefore $t$ is a well defined continuous map from $Z$ to $L$. It is a right inverse of the bundle projection by construction, and thus a section of $\calL$.

Take $\xi_1$ and $\xi_2$ in $\Gamma_0(\V)$, and let $t_1$ and $t_2$ be the sections of $\calL$ as constructed above. Write $\psi_i^{-1}\circ\xi_1(x)=(x,v_1(x))$ and $\psi_i^{-1}\circ\xi_2(x)=(x,v_2(x))$. We have 
\begin{align*}\langle t_1,t_2\rangle(x)&=\sum_{z\in s^{-1}(x)}\langle t_1,t_2\rangle_\calL(z)=\sum_{k=1}^n\overline{k_i^{-1}(t_1(h_i(x,k)))}k_i^{-1}(t_2(h_i(x,k)))\\&=\sum_{k=1}^n \overline{v_1(x)}_kv_2(x)_k=\langle\xi_1,\xi_2\rangle_\V(x) \end{align*}
for any $x\in U_i$, where we choose the metrics on $\calL$ and $\V$ to be compatible with the atlases $(s^{-1}(U_i),k_i)$ and $(U_i,\psi_i)$ as in Remark \ref{rem:metric atlas}. 

This shows that if we define $\Phi$ to send a section $\xi$ of $\V$ to the section $t$ of $\calL$ as constructed above, then we obtain an inner product preserving map from $\Gamma_0(\V)$ to $\Gamma_0(Z,s,\calL)$. The description of $\Phi$ in local coordinates, together with a partition-of-unity argument, yields that $\Phi$ respects $C_0(X)$-multiplication. A similar consideration shows that it is surjective. We conclude that $\Phi$ is an isomorphism from $\Gamma_0(\mathcal{P}\times\C^n)$ to $\Gamma_0(Z,s,\calL)$. 

Now assume that the twisted coverings $(Z,s,\calL)$ and $(Z',s',\calL')$ are isomorphic. This is equivalent to the corresponding $\calU(N_{D_n})$-bundles $\mathcal{P}$ and $\mathcal{P}'$ being isomorphic. Choose atlases as in Remark \ref{rem:compatible atlases}. We adopt the same notation, and add a prime to anything related to $\mathcal{P}'$. The isomorphism from $\mathcal{P}$ to $\mathcal{P}'$ is implemented by functions $r_i:U_i\to \calU(N_{D_n})$ such that $g'_{ij}(x)=r_i(x)^{-1}g_{ij}(x)r_j(x)$ for all $x\in U_i\cap U_j$ (see \cite[Theorem 2.7, Chapter 4]{Husemoller:1993}). These functions are also what implements the isomorphism from $\mathcal{M\coloneqq }\mathcal{P}\times(\{1,\dots,n\}\times\C)$ to $\mathcal{M}'\coloneqq\mathcal{P}'\times(\{1,\dots,n\}\times\C)$, as well as the isomorphism from $\V$ to $\V'$. The isomorphism from $\mathcal{Z}$ to $\mathcal{Z}'$ is implemented by $\tilde{r}_i\coloneqq \rho\circ r_i:U_i\to S_n$, where $\rho:\calU(N_{D_n})\to S_n$ is again the quotient map.

Take a section $\xi:X\to V$, and write $\xi_i$ for its restriction to $U_i$. Let $v(x)\in\C^n$ be such that $\psi_i^{-1}\circ\xi_i(x)=(x,v(x))$, for $x\in U_i$. Then the section $\xi_i'$ obtained from the isomorphism between $\V$ and $\V'$ fulfills $\psi'^{-1}_i\circ\xi_i'(x)=(x,r_i(x)^{-1}v(x))$. Thus going along the upper horizontal and then the right vertical arrow of Diagram (\ref{diag:module iso}) yields a section $t\in\Gamma_0(Z',s',\calL')$, that on $s'^{-1}(U_i)$ is given by $t'_i=\phi_i'\circ\tau_i'\circ h_i'^{-1}$ with $\tau_i'(x,k)=(x,k,(r_i(x)^{-1}v(x))_k)$. 

Going along the left vertical and then the lower horizontal arrow of Diagram (\ref{diag:module iso}) means to first construct $t\in\Gamma_0(Z,s,\calL)$ from $\xi$, and then obtain an element $t''$ of $\Gamma_0(Z',s',\calL')$ from $t$ by using the isomorphisms between $\mathcal{Z}$ and $\mathcal{Z}'$ as well as $\mathcal{M}$ and $\mathcal{M}'$. In local coordinates, we have $\phi_i'\circ t_i''\circ h_i'^{-1}=r_i^{-1}\circ\tau_i\circ\tilde{r}_i$. Calculate
\begin{align*}
    r_i^{-1}\circ\tau_i\circ\tilde{r}_i(x,k)&=r_i^{-1}\circ\tau_i(x,\tilde{r}_i(x)(k))=(x,r_i(x)^{-1}\cdot(\tilde{r}_i(x)(k),v(x)_{\tilde{r}_i(x)(k)}))\\&=(x,k,(r_i(x)^{-1}v(x))_k)=\tau_i'(x,k).
\end{align*}
This shows $t_i'=t_i''$, for any $i\in I$. Thus the diagram commutes. 
\end{proof}
\section{Correspondences}\label{sect:correspondences}

The purpose of this section is to establish the upper left arrow in Diagram (\ref{diag:triangle}), that is, given a sufficiently nice $C^*$-correspondence from $C_0(Y)$ to $C_0(X)$, we will establish a bijection between isomorphism classes of twisted $Y$-$X$ correspondences inducing $\E$, and isomorphism classes of certain atlases of the vector bundle associated to $\E$. 

If $A$ and $B$ are $C^*$-algebras, then a \emph{$C^*$-correspondence from $B$ to $A$} is a right Hilbert $A$-module $\E$ together with a $\ast$-homomorphism $\varphi$ from $B$ to $\calL(\E)$, the adjointable $A$-linear operators on $\E$. The $\ast$-homomorphism $\varphi$ is called the \emph{left action}.

Let $X$ and $Y$ be locally compact Hausdorff spaces. Throughout this section, $\E$ is a proper and nondegenerate $C^*$-correspondence from $C_0(Y)$ to $C_0(X)$ that is $\sigma$-finitely generated and $\sigma$-projective as a right $C_0(X)$-module. Proper means that the left action of $\E$, which we always denote by $\varphi$, takes values in the compact operators $\K(\E)$ of $\E$. Nondegenerate means that the linear span of $\varphi(C_0(Y))\E$ is dense in $\E$. Finally, a Hilbert $C_0(X)$-module is $\sigma$-finitely generated and $\sigma$-projective if its restriction to every compact subset $K$ of $X$ is finitely generated and projective. This is equivalent to the existence of a vector bundle $\V$ over an open subset of $X$ such that $\Gamma_0(\V)$ is isomorphic to $\E$ as right Hilbert $C_0(X)$-modules, see \cite[Theorem 3.4]{partI}.

\begin{remark}\label{rem:constant rank}
    In virtually all proofs in the following two sections, we can assume without loss of generality that the vector bundle $\V$ is defined on all of $X$. Otherwise, replace $X$ with the open set on which $\V$ is defined. This is equivalent to $\E$ being right full, meaning that the linear span of $\langle\E,\E\rangle$ is equal to $C_0(X)$. If the Hilbert module is constructed from a twisted covering $(Z,s,\calL)$ as in Definition \ref{def:twisted local homeo module}, then fulness is equivalent to $s$ being surjective. Furthermore it is sufficient to consider the case that the vector bundle $\V$ has constant rank. For if it does not, then because of local triviality we can partition the space $X$ into countably many clopen sets over which the rank of $\V$ is constant. We can then run the arguments on each of the clopen sets separately. Every point $x$ of $X$ has a neighborhood fully contained in the clopen set to which $x$ belongs. Since we will only be concerned with properties that should hold either pointwise or on small neighborhoods, this will prove the statements in general. 
\end{remark}
The following is a twisted version of the topological correspondences from \cite[Definition 1.2]{katsura:2004_2}. In particular, if the line bundle $\calL$ is trivial, we precisely recover \cite[Definition 1.2]{katsura:2004_2}.
\begin{definition}\label{def:topological correspondence}
Let $X$ and $Y$ be locally compact Hausdorff spaces. A \emph{twisted topological $Y$-$X$-correspondence} is a quadrupel $(Z,r,s,\calL)$ consisting of a twisted local homeomorphism $(Z,s,\calL)$ in the sense of Definition \ref{def:twisted local homeo}, and a continuous map $r$ from $Z$ to $Y$ called the \emph{range map}. In this context, $s$ is called the \emph{source map}. We say that $(Z,r,s,\calL)$ is \emph{left proper} if $r$ is proper and \emph{right proper} if $s$ is proper onto its image. The latter implies that $s$ is a covering map. A proper twisted $Y$-$X$-correspondence is both left and right proper.

Recall that in Definition \ref{def:twisted local homeo module} we associated a right Hilbert $C_0(X)$-module $\Gamma_0(Z,s,\calL)$ to a twisted covering $(Z,s,\calL)$. Given a twisted topological $Y$-$X$ correspondence $(Z,r,s,\calL)$, we define the $C^*$-correspondence $\Gamma_0(Z,r,s,\calL)$ from $C_0(Y)$ to $C_0(X)$ to be $\Gamma_0(Z,s,\calL)$ as a right Hilbert $C_0(X)$-module, with left action
\[
\varphi(f)\xi(z)=f(r(z))\xi(z),
\]
for $z\in Z$, $f\in C_0(Y)$ and $\xi\in \Gamma(Z,r,s,\calL)$.
\end{definition}
\begin{remark} 

    The twisted topological correspondence $(Z,r,s,\calL)$ is left proper if and only if the $C^*$-correspondence $\Gamma_0(Z,r,s,\calL)$ is proper.

   Topological spaces can be considered as étale groupoids with only identity arrows. A topological $Y$-$X$-correspondence (without the twist) is the same as a groupoid correspondence from $X$ to $Y$ in the sense of \cite[Definition 3.1]{antunes:2022}. It is left proper in the sense of Definition \ref{def:topological correspondence} if and only if it is proper as a groupoid correspondence in the sense of \cite[Definition 3.3]{antunes:2022}.
   
   If we take $X$ and $Y$ to be the same space, then Definition \ref{def:topological correspondence} yields the definition of twisted topological graphs as in \cite[Definition 3.2]{Li:2017}, also see \cite[Section 2]{KumjianLi:2017}. 
\end{remark}

Let $\E$ be a $C^*$-correspondence from $C_0(Y)$ to $C_0(X)$. We call a twisted $Y$-$X$-correspondence $(Z,r,s,\calL)$ \emph{associated to} $\E$ if $\Gamma_0(Z,r,s,\calL)$ and $\E$ are isomorphic as $C^*$-correspondences. We will also say that $\E$ is associated to $(Z,r,s,\calL)$. Any $C^*$-correspondence associated to a proper twisted topological $Y$-$X$-correspondence is proper, nondegenerate, $\sigma$-finitely generated and $\sigma$-projective. 

\subsection{Obtaining an atlas from a correspondence}\label{sect:atlas from corr}

Let $(Z,s,\calL)$ be a twisted $n$-sheeted covering of $X$ and let $\V$ be the vector bundle associated to $\Gamma_0(Z,s,\calL)$. We have seen in Theorem \ref{thm:twisted coverings and principal bundles} and Proposition \ref{prop:module iso} that there exists a principal $\calU(N_{D_n})$-bundle $\mathcal{P}$ such that $\Gamma_0(\mathcal{P}\times\C^n)$ is isomorphic to $\Gamma_0(Z,s,\calL)$, as a right Hilbert $C_0(X)$-module. Hence $\V$ is isomorphic to $\mathcal{P}\times\C^n$, and thus has an atlas whose transition functions take values in $\calU(N_{D_n})$. The twisted covering can be reconstructed from the principal bundle, which is determined up to isomorphism by the transition functions.

Generalizing from twisted coverings to twisted topological correspondences $(Z,r,s,\calL)$, we might therefore expect that the vector bundle of the $C^*$-correspondence $\Gamma_0(Z,r,s,\calL)$ has an atlas which somehow encodes the underlying topological correspondence. However, we have to ask more from this atlas then just having $\calU(N_{D_n})$-valued transition functions, since we need to incorporate the left action. In Definition \ref{def:diagonalizing atlas} we will identify the key property of such atlases. 

In Proposition \ref{prop:correspondence to atlas} and Theorem \ref{thm:corr and atlas one to one}, we will consider vector bundles of nonconstant rank. The following definition is meant to avoid cumbersome notation in such cases. 

\begin{definition}\label{def:normalizing atlas}
Let $\V$ be a vector bundle, let $(U_i,h_i)_{i\in I}$ be an atlas of $\V$, and write $g_{ij}$ for the transition functions. Then the atlas is called \emph{normalizing} if $g_{ij}(x)$ lies in $\calU(N_{D_{n_x}})$ for all $x\in U_i\cap U_j$ and all $i,j\in I$. 
\end{definition}
 
\begin{remark}\label{rem:frame}
Let $\V$ be a vector bundle, and let $(U,h)$ be a chart of $\V$. Assume that the closure of $U$ is compact. Choose any nonnegative function $\gamma$ whose open support is equal to $U$. Write $e_k$ for the $k$-th unit vector in $\C^n$, where $n$ is the rank of $\V$ over $U$. Then setting $\xi_{k}(x)\coloneqq h(x,\gamma(x) e_k)$ defines sections $\xi_k\in\Gamma_0(\V)$, for all $k\in\{1,\dots,n\}$. It follows from this definition that $(\xi_1(x),\dots,\xi_n(x))$ is an orthogonal basis for the fiber $\V_x$ over any point $x$ in $U$. 

If $(U_i,h_i)_{i\in I}$ is an atlas for $\V$, then we obtain sections $\xi_{ik}$ for all $i\in I$. The set $\{\xi_{ik}\}_{i,k}$ forms a frame for $\Gamma_0(\V)$, see \cite[Lemma 2.6]{partI}. This means that for all $\eta\in \E$ we have 
\[\eta=\sum_{i,k}\xi_{ik}\langle\xi_{ik},\eta\rangle\]
where the sum, if it is infinite, converges in norm. 
\end{remark}

We now define the property that makes it possible to keep track of the left action when considering atlases of $\V$. A version of this appeared in \cite[Corollary 6.5]{KaliszewskiPataniQuigg:2013}.
\begin{definition}\label{def:diagonalizing atlas}
Let $(U,h)$ be a chart of $\V$ such that the closure of $U$ is compact. Then we say that the chart $(U,h)$ \emph{diagonalizes the left action} if $h(x,e_k)$ is an eigenvector of $\varphi(f)_x$ for all $x\in U$, all $k$ and all $f\in C_0(Y)$.

If $(U_i,h_i)_{i\in I}$ is an atlas of $\V$ such that each $U_i$ has compact closure, then we say that it diagonalizes the left action if it consists of charts diagonalizing the left action.  
 \end{definition}
 Note that a chart $(U,h)$ of $\V$ diagonalizes the left action if and only if $\xi_k(x)$ is an eigenvector of $\varphi(f)_x$, for all $x\in U$, all $k$ and all $f\in C_0(Y)$.
 
Let $\mathcal{P}$ be the a principal $\calU(N_{D_n})$-bundle. Given an atlas of $\mathcal{P}$, denote by $(U_i,\phi_i)_{i\in I}$ the induced atlas of the associated bundle $\mathcal{P}\times\C^n$. We assume that all $U_i$ have compact closure. Given a partition of unity $\{\gamma_i\}_{i\in I}$ subordinate to $\{U_i\}_{i\in I}$, we can define sections $\xi_{ik}$ as in Remark \ref{rem:frame}.

Let $(Z,s,\calL)$ be the twisted covering associated to $\mathcal{P}$. The atlas of $\mathcal{P}$ induces an atlas of $(Z,s,\calL)$ over the same open sets, by which we mean base-preserving trivializations $h_i:U_i\times\{1,\dots,n\}\to s^{-1}(U_i)$ and $k_i:s^{-1}(U_i)\times\C\to p_L^{-1}(s^{-1}(U_i))$. Given the partition of unity $\{\gamma_i\}_{i\in I}$ from the previous paragraph, define functions $\gamma_{ik}$ from $U_i\times\{1,\dots,n\}$ to $[0,1]$ by $\gamma_{ik}(x,l)=\delta_{lk}\gamma_i(x)$, where $x\in U_i$ and $\delta_{lk}$ is the Kronecker delta. In words, $\gamma_{ik}$ equals $\gamma_i$ on the $k$-th copy of $U_i$, and is zero everywhere else. For any $y\in s^{-1}(U_i)$, define $\Tilde{\xi}_{ik}(y)=k_i(y,\gamma_{ik}(h_i^{-1}(y))$. Then $\tilde{\xi}_{ik}$ is a continuous section of $\calL$ vanishing at infinity, and hence an element of $\Gamma_0(Z,s,\calL)$.  

\begin{lemma}\label{lem:canonical sections}
    The isomorphism between $\Gamma_0(\mathcal{P}\times\C^n)$ and $\Gamma_0(Z,s,\calL)$ from Proposition \ref{prop:module iso} sends $\xi_{ik}$ to $\tilde{\xi}_{ik}$.
\end{lemma}
\begin{proof}
    This is obvious if one considers what the isomorphism from $\Gamma_0(\mathcal{P}\times\C^n)$ to $\Gamma_0(Z,s,\calL)$ looks like in local coordinates, which was described in the proof of Proposition \ref{prop:module iso}.
\end{proof}
\begin{proposition}\label{prop:correspondence to atlas}
Let $(Z,r,s,\calL)$ be a proper twisted $Y$-$X$-correspondence associated to $\E$. Then $\V$, the vector bundle such that $\Gamma_0(\V)\cong\E$ as right Hilbert $C_0(X)$-modules, has a normalizing atlas diagonalizing the left action. 
\end{proposition}
\begin{proof}

We assume that $\V$ has constant rank $n$, or equivalently that $(Z,s,\calL)$ is a twisted $n$-sheeted covering. Then Theorem \ref{thm:twisted coverings and principal bundles} yields a principal $\calU(N_{D_n})$-bundle such that $\Gamma_0(Z,s,\calL)$ and $\Gamma_0(\mathcal{P}\times\C^n)$ are isomorphic as right Hilbert $C_0(X)$-modules. Using this isomorphism we can define a left action on $\Gamma_0(\mathcal{P}\times\C^n)$ making it a $C^*$-correspondence isomorphic to $\E$. To ease notation we will not distinguish between the left actions on $\Gamma_0(\mathcal{P}\times\C^n)$ and $\E$, and write $\varphi$ for both of them. Note that $\V$ is isomorphic to $\mathcal{P}\times\C^n$, and hence a normalizing atlas of $\mathcal{P}\times\C^n$ diagonalizing the left action will give rise to a normalizing atlas of $\V$ diagonalizing the left action of $\E$. 

Take an atlas of $\mathcal{P}$, and let $(U_i,\psi_i)_{i\in I}$ be the induced atlas of the associated bundle $\mathcal{P}\times\C^n$. From Remark \ref{rem:compatible atlases} we obtain an atlas of $(Z,s,\calL)$ over the same open sets, meaning  fiber-preserving trivializations $h_i:U_i\times\{1,\dots,n\}\to s^{-1}(U_i)$ and $k_i:s^{-1}(U_i)\times\C\to p_L^{-1}(s^{-1}(U_i))$. Construct the sections $\xi_{ik}\in\Gamma_0(\mathcal{P}\times\C^n)$ and $\tilde{\xi}_{ik}\in\Gamma_0(Z,s,\calL)$ as above. We know from Lemma \ref{lem:canonical sections} that the isomorphism between $\Gamma_0(Z,s,\calL)$ and $\Gamma_0(\mathcal{P}\times\C^n)$ maps $\xi_{ik}$ to $\tilde{\xi}_{ik}$, for all $i\in I$ and $k=1,\dots,n$. 

Take $x\in U_i$ and $y\in s^{-1}(x)$. We have 
\[\varphi(f)\Tilde{\xi}_{ik}(y)=f(r(y))k_i(y,\gamma_{ik}(h_i^{-1}(y))).\]
By definition of $\gamma_{ik}$, the right hand side is only nonzero if $h_i^{-1}(y)=(x,k)$, in which case it equals $f(r(h_i(x,k)))k_i(y,\gamma_i(x))$. This yields
\[\varphi(f)\Tilde{\xi}_{ik}(y)=f(r(h_i(x,k)))\Tilde{\xi}_{ik}(y).\]
Note that $f(r(h_i(x,k)))$ is independent of which $y\in s^{-1}(x)$ we choose. Hence, expressing the equality in terms of the sections $\xi_{ik}$, which we are allowed to do if we identify $\Gamma_0(\mathcal{P}\times\C^n)$ with $\Gamma_0(Z,s,\calL)$ by virtue of Lemma \ref{lem:canonical sections}, we obtain
\[\varphi(f)_x\xi_{ik}(x)=(\varphi(f)\xi_{ik})(x)=f(r(h_i(x,k)))\xi_{ik}(x)\]
for an arbitrary $f\in C_0(X)$. 
In particular, $\xi_{ik}(x)$ is an eigenvector of $\varphi(f)_x$ for every $f\in C_0(X)$. This shows that the atlas $(U_i,\psi_i)_{i\in I}$ diagonalizes the left action. 
\end{proof}

\subsection{Obtaining a topological correspondence from an atlas} \label{sec:corr from atlas}
Let $\E$ be a proper, nondegenerate, $\sigma$-finitely generated and $\sigma$-projective $C^*$-correspondence from $C_0(Y)$ to $C_0(X)$ such that the associated vector bundle $\V$ has a normalizing atlas diagonalizing the left action. We assume that $\V$ has constant rank $n$, see Remark \ref{rem:constant rank}. Then $\V$ is of the form $\mathcal{P}\times\C^n$ for a principal $\calU(N_{D_n})$-bundle $\mathcal{P}$. Write $\varphi$ for the left action on $\E=\Gamma_0(\mathcal{P}\times\C^n)$.

Theorem \ref{thm:twisted coverings and principal bundles} yields a twisted $n$-sheeted covering $(Z,s,\calL)$ such that $\Gamma_0(Z,s,\calL)$ and $\Gamma_0(\mathcal{P}\times\C^n)$ are isomorphic Hilbert $C_0(X)$-modules. To obtain a topological correspondence, we require a range map. 

\begin{remark}\label{rem:gelfand}
    In the following proposition, and in several other places throughout the paper, we will use a non-unital version of Gelfand--Naimark (see \cite{nlab:gelfand_duality}). It asserts that there is an equivalence of categories between the category of commutative non-unital $C^*$-algebras with nondegenerate $\ast$-homomorphisms, and the category of locally compact Hausdorff spaces and proper maps. Here a $\ast$-homomorphism is called nondegenerate if the two-sided ideal generated by its image is dense in the codomain. 

    If the spaces we deal with happen to be compact, then we can use the standard formulation of the Gelfand--Naimark theorem for unital $C^*$-algebras. 
\end{remark}

\begin{proposition}\label{prop:constructing range map}
There exists a unique proper continuous map $r$ from $Z$ to $Y$ such that the $C^*$-correspondences $\E$ and $\Gamma_0(Z,r,s,\calL)$ are isomorphic. 
\end{proposition}
\begin{proof}

By assumption, $\V=\mathcal{P}\times\C^n$ has a normalizing atlas $(U_i,\phi_i)$ diagonalizing the left action. Let $\xi_{ik}$ be sections as in Definition \ref{def:diagonalizing atlas}. Then $\xi_{ik}(x)$ is an eigenvector of $\varphi(f)_x$ for any $x\in U_i$ and $f\in C_0(Y)$. 

For any $f\in C_0(X)$, define a map $\mu_{ik}(f)$ to send $x\in U_i$ to the eigenvalue of $\varphi(f)_x$ associated to $\xi_{ik}(x)$. This means we have $\varphi(f)_x\xi_{ik}=\mu_{ik}(f)(x)\xi_{ik}(x)$. The eigenvalues of $\varphi(f)_x$ depend continuously on $x$, and hence $\mu_{ik}(f)$ is an element of $C_b(U_i)$. 

As in Remark \ref{rem:compatible atlases}, the atlas $(U_i,\phi_i)$ of $\mathcal{P}\times\C^n$ yields an atlas $(U_i,h_i)$ trivializing the covering $s:Z\to X$. Define a map $\psi$ from $C_0(Y)$ to $C_0(Z)$  by sending $f\in C_0(Y)$ to the function on $Z$ that sends $z\in h_i(U_i\times\{k\})$ to $\psi(f)(z)=\mu_{ik}(f)(s(z))$. To show that this is well defined, we need to show that if $z$ lies both in $ h_i(U_i\times\{k\})$ and in $ h_j(U_j\times\{l\})$ for $i$ and $j$ such that $U_i\cap U_j\neq\emptyset$, then $\mu_{ik}(f)(s(z))=\mu_{jl}(f)(s(z))$. Write $x\coloneqq s(z)$. We have $l=\sigma_{ji}(x)k$, where $\sigma_{ji}:U_i\cap U_j\to S_n$ is the transition function of the covering $s:Z\to X$, see Appendix \ref{appendix A}.

Then 
\[\xi_{ik}(x)=\phi_i(x,\gamma_i(x)e_k)=\phi_j(x,\gamma_i(x)g_{ji}(x)e_k),\]
where $g_{ji}=\phi_j^{-1}\circ\phi_i$ is the transition function. Since $g_{ji}(x)$ is an element of $\calU(N_{D_n})$, it is a product of a permutation matrix with a unitary diagonal matrix. By Remark \ref{rem:compatible atlases}, the permutation matrix we obtain from $g_{ij}$ is exactly the matrix associated to $\sigma_{ji}(x)$. Hence we have $g_{ji}(x)e_k=\lambda e_l$ for a suitable $\lambda\in\T$. This yields
\begin{align*}\xi_{ik}(x)=\lambda \phi_j(x,\gamma_i(x)e_l)=\lambda\gamma_j(x)^{-1}\gamma_i(x)\phi_j(x,\gamma_j(x)e_l)=\lambda\gamma_j(x)^{-1}\gamma_i(x)\xi_{jl}(x).\end{align*}
Recall that by assumpion the open support of $\gamma_j$ is equal to $U_j$, which means that the expression $\gamma_j(x)^{-1}$ makes sense. Hence
\begin{align*}
\mu_{ik}(f)(x)\xi_{ik}(x)&=\varphi(f)_x\xi_{ik}(x)=\lambda_k\gamma_j(x)^{-1}\gamma_i(x)\varphi(f)_x\xi_{jl}(x)\\&=\lambda\gamma_j(x)^{-1}\gamma_i(x)\mu_{jl}(f)(x)\xi_{jl}(x)\\&=\mu_{jl}(f)(x)\xi_{ik}(x).
\end{align*}
Since $\xi_{i,k}(x)\neq 0$ this implies $\mu_{ik}(f)(s(z))=\mu_{jl}(f)(s(z))$. Hence, $\psi(f)$ is a well defined complex-valued map on $Z$, for every $f\in C_0(X)$. It is continuous, since its restriction to each $h_i(U_i\times\{k\})$ is. That $\psi$ is a $\ast$-homomorphism also follows from the fact that it is locally. 

Since $\E$ is proper, $\varphi(f)$ is a compact operator for every $f\in C_0(Y)$. Hence for every $\varepsilon>0$ there exists a compact set $K$ such that the norm of $\varphi(f)_x$ is less then $\varepsilon$, for all $x$ not in $K$ (see \cite[Proposition 3.5]{partI}. Hence all eigenvalues of $\varphi(f)_x$ have modulus smaller then $\varepsilon$. We obtain that $\psi(f)$ vanishes at infinity. Furthermore, $\psi$ is nondegenerate. This follows from nondegeneracy of $\E$. 

Therefore, $\psi$ is is a well defined nondegenerate $\ast$-homomorphism from $C_0(Y)$ to $C_0(Z)$. By Gelfand duality we obtain a proper continuous map $r\coloneqq\phi^*$ from $Z$ to $Y$. 

Write $\varphi_r$ for the left action of $\Gamma_0(Z,r,s,\calL)$ as in Definition \ref{def:topological correspondence}. We regard $\varphi_r$ as a left action on $\Gamma_0(\mathcal{P}\times\C^n)$.

Let $x\in U_i$, and $k\in\{1,\dots,n\}$. By the proof of Proposition \ref{prop:correspondence to atlas} we have $\varphi_r(f)\xi_{ik}(x)=f(r(h_i(x,k)))\xi_{ik}(x)$. We obtain
\begin{align*}\varphi_r(f)\xi_{ik}(x)&=f(r(h_i(x,k)))\xi_{ik}(x)=\psi(f)(h_i(x,k))\xi_{i,k}(x)=\mu_{ik}(f)(x)\xi_{i,k}(x)\\&=\varphi(f)\xi_{i,k}(x).\end{align*}
This is enough to show $\varphi_r=\varphi$, since the $\xi_{ik}$ form a frame of $\Gamma_0(\mathcal{P}\times\C^n)$ (see Remark \ref{rem:frame}). 

Lastly, we want to show that $r$ is uniquely determined. Take $r':Z\to Y$ such that there exists $z\in Z$ with $r(z)\neq r'(z)$. Write $x=s(z)$ and take $i$ and $k$ such that $x\in U_i$ and $z=h_i(x,k)$. Then there exists $f\in C_0(Y)$ such that $f(r(z))\neq f(r'(z))$, which implies
\[\varphi_{r}(f)\xi_{ik}(x)=f(r(h_i(x,k)))\xi_{ik}(x)=f(r(z))\xi_{ik}(x)\neq f(r'(z))\xi_{ik}(x)=\varphi_{r'}(f)\xi_{ik}(x)\]
and hence $\varphi_r\neq\varphi_{r'}$. This finishes the proof. 
\end{proof} 

\subsection{Equivalence of atlases}\label{sect:equiv of atlases}

  Take two atlases $(U_i,h_i)$ and $(U_i,h_i')$ of vector bundles $\V$ and $\V'$, respectively, and let $g_{ij}$ and $g_{ij}'$ be their transition functions. Assume that the dimension of $\V_x$ is equal to the dimension of $\V_x'$ for all $x\in X$. We regard $g_{ij}$ and $g_{ij}'$ as functions from $U_i\cap U_j$ into $\mathcal{U}(\K)$ with $\K$ the compact operators on a separable Hilbert space. It is understood that $g_{ij}(x)$ and $g_{ij}'(x)$ lie in $\mathcal{U}(M_{n_x})$ for $x\in U_i\cap U_j$, where $n_x$ is the rank of $\V$ and $\V'$ over $x$. The systems of transition functions are called \emph{equivalent} if there exist continuous functions $r_i:U_i\to\mathcal{U}(\K)$, with $r_i(x)$ lying in $\mathcal{U}(M_{n_x})$ for all $x\in U_i$, such that $g_{ij}'(x)=r_i(x)^{-1}g_{ij}(x)k_j(x)$ holds for all $x\in U_i\cap U_j$, see \cite[Definition 2.6, Chapter 5]{Husemoller:1993}.  Any choice of such functions $r_i$ yields an isomorphism between $\V$ and $\V'$, see \cite[Theorem 2.7, Chapter 5]{Husemoller:1993}. This isomorphism is given by sending $v\in\V_{x}$ for $x\in U_i$ to $h_i'(x,r_i(x)^{-1}h_i^{-1}(x,v))$.

\begin{definition}\label{def:equivalence diagonalizing atlases}
Let $\E$ and $\E'$ be two proper nondegenerate $\sigma$-finitely generated and $\sigma$-projective $C^*$-correspondences from $C_0(Y)$ to $C_0(X)$, and let $\V$ and $\V'$ be vector bundles over $X$ such that $\Gamma_0(\V)\cong\E$ and $\Gamma_0(\V')\cong\E'$ as right Hilbert $C_0(X)$-modules. Assume that the dimension of $\V_x$ is equal to the dimension of $\V_x'$ for all $x\in X$. Let $(U_i,h_i)$ and $(U_i,h_i')$ be two normalizing atlases of $\V$ and $\V'$, respectively, diagonalizing the left actions of $\E$ and $\E'$. Then the atlases are called \emph{equivalent} if they have equivalent systems of transition functions, and if there exists a choice of functions $r_i$ implementing this equivalence such that $r_i(x)\in\calU(N_{D_{n_x}})$ for all $x\in U_i$, and such that the induced isomorphism between $\Gamma_0(\V)$ and $\Gamma_0(\V')$ intertwines the left actions $\varphi$ and $\varphi'$. 
\end{definition}

\begin{definition}
We say that two twisted $Y$-$X$-correspondences $(Z,r,s,\calL)$ and $(Z',r',s',\calL')$ are isomorphic if there exists a homeomorphism $\phi$ from $Z$ to $Z'$ fulfilling $s'\circ\phi=s$ and $r'\circ\phi=r$, and such that the pullback $\phi^*\calL'$ of $\calL'$ to $Z$ is isomorphic to $\calL$.
\end{definition}

The following theorem generalizes \cite[Corollary 6.2]{frausino2023} to the case of twisted correspondences, in particular taking the left action into account. It establishes the first arrow in Diagram (\ref{diag:triangle}). 

\begin{theorem}\label{thm:corr and atlas one to one}
The constructions from Propositions \ref{prop:correspondence to atlas} and \ref{prop:constructing range map} set up a bijection between isomorphism classes of proper twisted $Y$-$X$-correspondences $(Z,r,s,\calL)$ associated to $\E$, and equivalence classes of normalizing atlases of $\V$ diagonalizing the left action of $\E$. Here equivalence of atlases is in the sense of Definition \ref{def:equivalence diagonalizing atlases}. The twist $\calL$ is trivial if and only if there exists an atlas in the equivalence class whose transition functions take values in the permutation matrices.  
\end{theorem}
\begin{proof}
In view of Propositions \ref{prop:correspondence to atlas} and \ref{prop:constructing range map}, all we have left to show to establish the bijection is that two twisted $Y$-$X$-correspondences are isomorphic if and only if the associated atlases are equivalent. Without loss of generality (see Remark \ref{rem:constant rank}),we can restrict to the case that $\V$ has constant rank $n$. 

Let $(Z,r,s,\calL)$ and $(Z',r',s',\calL')$ be two isomorphic twisted $Y$-$X$-correspondences associated to $\E$. Let $\mathcal{P}$ and $\mathcal{P}'$ be the principal $\calU(N_{D_n})$-bundles corresponding to the twisted coverings $(Z,s,\calL)$ and $(Z',s',\calL')$ as in Theorem \ref{thm:twisted coverings and principal bundles}. By Theorem \ref{thm:twisted coverings and principal bundles} and Proposition \ref{prop:module iso} that $\mathcal{P}$ and $\mathcal{P}'$ are isomorphic, and that the induced isomorphism from $\mathcal{P}\times\C^n$ to $\mathcal{P}'\times\C^n$ makes Diagram (\ref{diag:module iso}) commute. 

Recall from Proposition \ref{prop:correspondence to atlas} how we obtained a normalizing atlas $(U_i,h_i)$ of $\mathcal{P}\times\C^n$ diagonalizing the left action. We chose an atlas $(U_i,\phi_i)$ of $\mathcal{P}$, and then showed that the induced atlas of $\mathcal{P}\times\C^n$ is indeed normalizing and diagonalizes the left action. Of course the same holds for $(Z',r',s',\calL')$ and $\mathcal{P}'$, with atlases $(U_i,\phi_i')$ and $(U_i,h_i')$ of $\mathcal{P}'$ and $\mathcal{P}'\times\C^n$, respectively. Now since $\mathcal{P}$ and $\mathcal{P}'$ are isomorphic, the atlases $(U_i,\phi)$ and $(U_i,\phi_i')$ must have equivalent systems of transition functions. Thus there exist functions $r_i$ from $U_i$ to $\calU(N_{D_{n}})$ implementing this equivalence.

The atlas $(U_i,h_i)$ has the same transition functions as $(U_i,\phi_i)$, and the same holds for $(U_i,h_i')$ and $(U_i,\phi_i')$. Hence we can use the functions $r_i$ to define an isomorphism from $\mathcal{P}\times\C^n$ to $\mathcal{P}'\times\C^n$ as in \cite[Theorem 2.7, Chapter 5]{Husemoller:1993}. The resulting map is exactly the same as the isomorphism from Proposition \ref{prop:module iso} that makes Diagram (\ref{diag:module iso}) commute. By commutativity of Diagram (\ref{diag:module iso}) the isomorphism from $\Gamma_0(\mathcal{P}\times\C^n)$ to $\Gamma_0(\mathcal{P}'\times\C^n)$ intertwines the left actions. This shows that $(U_i,h_i)$ and $(U_i,h_i')$ are equivalent atlases in the sense of Definition \ref{def:equivalence diagonalizing atlases}.

 Now let $(U_i,h_i)$ and $(U_i,h_i')$ be two normalizing atlases of the vector bundle associated to $\E$ that diagonalize the left action. Assume that they are equivalent in the sense of Definition \ref{def:equivalence diagonalizing atlases}. This implies that there is an isomorphism between the principal $\calU(N_{D_n})$-bundles $\mathcal{P}$ and $\mathcal{P}'$ such that the induced isomorphism from $\Gamma_0(\mathcal{P}\times\C^n)$ to $\Gamma_0(\mathcal{P}'\times\C^n)$ intertwines the left action. From Theorem \ref{thm:twisted coverings and principal bundles} and Proposition \ref{prop:module iso} we obtain an isomorphism of twisted coverings $(Z,s,\calL)$ and $(Z',s',\calL')$ making Diagram (\ref{diag:module iso}) commute. Now uniqueness of the range map in Proposition \ref{prop:constructing range map} yields that the twisted topological correspondences $(Z,r,s,\calL)$ and $(Z',r',s',\calL)$ are isomorphic. 

The last statement of the theorem, namely that the twist $\calL$ is trivial if and only if there exists an atlas in the equivalence class whose transition functions take values in the permutation matrices, follows from the according statement in Theorem \ref{thm:twisted coverings and principal bundles}.
\end{proof}
\section{Cartan subalgebras in the compacts}\label{sect:perspective of the compacts}
In this section we fill in the two missing arrows of Diagram (\ref{diag:triangle}), and show that the diagram commutes up to isomorphism.
\subsection{Cartan subalgebras and atlases}
Let $\E$ be a proper nondegenerate $\sigma$-finitely generated and $\sigma$-projective $C^*$-correspondence from $C_0(Y)$ to $C_0(X)$. Write $\varphi$ for its left action, and $\V$ for the vector bundle associated to $\E$ via \cite[Theorem 3.4]{partI}. In this section we show that specifying an atlas of $\V$ that is normalizing and diagonalizes the left action $\varphi$, is the same thing as specifying a Cartan subalgebra inside of $\K(\E)$ which contains the image of $\varphi$.

Let $\aut(M_n,D_n)$ denote the group of automorphisms $\alpha$ of $M_n$ such that $\alpha(D_n)=D_n$. Note that since all automorphisms of $M_n$ are inner, elements of this group are precisely those automorphisms which are implemented by a unitary in $\calU(N_{D_n})$. Explicitly, we have
\[\aut(M_n,D_n)=\{\Ad(u):u\in\calU(N_{D_n})\},\]
where $\Ad(u)$ denotes the inner automorphism induced by $u$.

If $\E$ is a $\sigma$-finitely generated and $\sigma$-projective Hilbert $C_0(X)$-module all of whose fibers have rank $n$, then $\K(\E)$ is a continuous $C_0(X)$-algebra with fibers isomorphic to $M_n$, see \cite[Section 3.2]{boenicke:2018}. In particular $\K(\E)$ is $n$-homogeneous, so we can apply Li's and Renault's work from \cite{LiRenault:2019}. Any $n$-homogeneous $C^*$-algebra can be viewed as the space of sections of a fiber bundle with fibers $M_n$, structure group $\aut(M_n)$ and base space the primitive ideal space. There is a principal $\aut(M_n)$-bundle associated to this fiber bundle, and the content of \cite[Proposition 2.2]{LiRenault:2019} is that this principal bundle reduces to a principal $\aut(M_n,D_n)$-bundle if and only if the $n$-homogeneous $C^*$-algebra has a Cartan subalgebra.

We obtain that $K(\E)$ has a Cartan subalgebra if and only if the associated principal $\aut(M_n)$-bundle reduces to an $\aut(M_n,D_n)$-bundle. Since we have pointed out above that $\aut(M_n,D_n)$ is the group of all such automorphisms which are implemented by elements of $\calU(N_{D_n})$, it is perhaps not surprising that this happens precisely when the vector bundle associated to $\E$ has an atlas whose transition function take values in $\calU(N_{D_n})$.
This is the content of the next proposition. 
\begin{remark}\label{rem:transition functions vector and mn bundle}
 Let $\V$ be a rank $n$ vector bundle over $X$, and write $\mathcal{B}$ for the fiber bundle such that $\K(\Gamma_0(\V))=\Gamma_0(\mathcal{B})$. Note that the fiber of $\mathcal{B}$ over a point $x\in X$ is given by $\mathrm{End}(\V_x)\cong M_n$. Let $(U_i,h_i)$ be an atlas of $\V$ with transition functions $g_{ij}:U_i\cap U_j\to \mathcal{U}(M_n)$. We define $H_i:U_i\times M_n\to \mathcal{B}_{U_i}$ by $H_i(x,a)v=h_i(x,ah_i^{-1}(x,v))$ for $v\in \V_x$. Then $(U_i,H_i)$ is an atlas of $\mathcal{B}$ with transition functions $\tilde{g}_{ij}:U_i\cap U_j\to \aut(M_n)$ given by $\tilde{g}_{ij}(x)=\Ad(g_{ij}(x))$.
 \end{remark}
\begin{proposition}\label{prop:transition functions vector and mn bundle}
If the transition functions $g_{ij}$ of an atlas of $\V$ take values in $\calU(N_{D_n})$, then the transition functions $\Ad(g_{ij})$ of the induced atlas of $\mathcal{B}$ as in Remark \ref{rem:transition functions vector and mn bundle} take values in $\aut(M_n,D_n)$. Every atlas of $\mathcal{B}$ whose transition functions take values in $\aut(M_n,D_n)$ arises this way. 
\end{proposition}
\begin{proof}
It is clear from the definitions that if $g_{ij}$ takes values in $\calU(N_{D_n})$, then $\Ad(g_{ij})$ takes values in $\aut(M_n,D_n)$.

Assume we are given an atlas of $\mathcal{B}$ over some open sets $U_i$ whose transition functions $\tilde{h}_{ij}$ take values in $\aut(M_n,D_n)$. Take any atlas of $\V$ with transition functions $g_{ij}$, and denote the induced transition functions of $\mathcal{B}$ by $\Tilde{g}_{ij}$. We can assume that this atlas is over the same open sets $U_i$, otherwise we take intersections. By \cite[Theorem 4.1, Chapter 6]{Husemoller:1993}, there exist functions $\tilde{r}_i:U_i\to\aut(M_n)$ such that $\Tilde{h}_{ij}(x)= \tilde{r}_i(x)^{-1}\Tilde{g}_{ij}(x)\tilde{r}_j(x)$ for all $x\in U_i\cap U_j$. Let $r_{i}(x)$ be unitaries implementing $\tilde{r}_i(x)$, and let $h_{ij}(x)$ be unitaries in $\calU(N_{D_n})$ implementing $\Tilde{h}_{ij}(x)$. Combining the short exact sequence (4.26) from \cite{raeburnwilliams:1998} with \cite[Lemma 4.33]{raeburnwilliams:1998} shows that by passing to a refinement of the cover $\{U_i\}_{i\in I}$, we can assume the functions $x\mapsto r_i(x)$ to be continuous. 

Applying $\Tilde{h}_{ij}(x)$ to some $a\in M_n$ yields \[h_{ij}(x)ah_{ij}(x)^*=r_{i}^*g_{ij}r_{j}(x)ar_{j}^*g_{ij}^*r_{i}(x).\]
    This holds for all $a$ in $M_n$, and hence $h_{ij}^*r_{i}^*g_{ij}r_{j}(x)$ lies in the center of $M_n$. Defining $\lambda_{ij}(x)\coloneqq h_{ij}^*r_{i}^*g_{ij}r_{j}(x)$ we obtain a  function $\lambda_{ij}$ from $U_i\cap U_j$ into $\T$. Using Theorem 2.7 in Chapter 5 of \cite{Husemoller:1993}, we see that the $\lambda_{ij}h_{ij}=r_i^*g_{ij}r_j$ are the transition functions of some atlas of $\V$. We have $\Ad(\lambda_{ij}h_{ij})=\Ad(h_{ij})=\tilde{h}_{ij}$, which means that we have found an atlas of $\V$ whose transition functions take values in $\calU(N_{D_n})$, and which induces the atlas of $\mathcal{B}$ that we started with. 
\end{proof}
\begin{proposition}\label{prop:cartan and atlas}
The vector bundle $\V$ has a normalizing atlas diagonalizing the left action $\varphi$ of $\E=\Gamma_0(\V)$ if and only if $\K(\E)$ has a Cartan subalgebra containing the image of $\varphi$.
\end{proposition}
\begin{proof}
As outlined in Remark \ref{rem:constant rank}, we can assume without loss of generality that the rank of the fiber $\V_x$ is equal to a constant $n\in\N$, for all $x\in X$. Let $(U_i,h_i)_{i\in I}$ be a normalizing atlas diagonalizing the left action of $\E$. Remark \ref{rem:transition functions vector and mn bundle} and Proposition \ref{prop:transition functions vector and mn bundle} give an atlas $(U_i,H_i)_{i\in I}$ for $\mathcal{B}$, the principal $M_n$-bundle such that $\K(\Gamma_0(\V))\cong\Gamma_0(\mathcal{B})$, whose transition functions take values in $\aut(M_n,D_n)$. Hence \cite[Proposition 2.2]{LiRenault:2019} shows that $\K(\Gamma_0(\V))$ has a Cartan subalgebra $D$. By construction, $D$ is the space of continuous sections of a subbundle $\mathcal{D}$ of $\mathcal{B}$. We have
\begin{equation}\label{eq:cartan in compacts}\mathcal{D}_{U_i}=H_i(U_i\times D_n)\qquad\text{ for all } i\in I.\end{equation} 

The atlas $(U_i,h_i)$ diagonalizes the left action, and so $H_i^{-1}(\varphi(f)_x)$ lies in $D_n$ for all $x\in U_i$ and all $f\in C_0(Y)$. Together with Equation (\ref{eq:cartan in compacts}), this shows that the image of $\varphi$ is contained in $D$.

The argument can be run backwards in order to show that if we start with a Cartan in $\K(\E)$ containing the image of $\varphi$, then we obtain a normalizing atlas of $\V$ diagonalizing the left action. 
\end{proof}
\begin{definition}\label{def:conjugacy of cartans}
Let $\E$ be a $C^*$-correspondence over $C_0(X)$. We call two Cartan subalgebras $D$ and $D'$ in $\K(\E)$ \emph{conjugate} if there exists an automorphism of $\E$ such that the induced automorphism of $\K(\E)$ maps $D$ to $D'$. Here an automorphism of $\E$ is a bijective inner-product preserving $C_0(X)$-linear map on $\E$ that commutes with the left action. This is more restrictive than the usual notion of automorphism of $C^*$-correspondences, which would allow for a $\ast$-automorphism on $C_0(X)$.
\end{definition}
\begin{theorem}\label{thm:cartan and atlas}
The construction from Proposition \ref{prop:cartan and atlas} sets up a bijection between conjugacy classes of Cartan subalgebras of $\K(\E)$ containing the image of $\varphi$, and equivalence classes of normalizing atlases of $\V$ diagonalizing the left action of $\E$.
\end{theorem}
\begin{proof}
We again assume that $\V$ has constant rank $n$. 

It follows from the constructions that if we start with a normalizing atlas of $\V$ that diagonalizes the left action and employ Proposition \ref{prop:cartan and atlas} to obtain a Cartan subalgebra, then using the other direction of Proposition \ref{prop:cartan and atlas} gives back the atlas we started with. Hence all we need to show is that two normalizing atlases which diagonalize the left action are equivalent in the sense of Definition \ref{def:equivalence diagonalizing atlases} if and only if the associated Cartan subalgebras of $\K(\E)$ are conjugate in the sense of Definition \ref{def:conjugacy of cartans}. 

Assume that two atlases $(U_i,h_i)$ and $(U_i,h_i')$ with associated Cartan subalgebras $D$ and $D'$ are equivalent. By definition this means that there exist functions $r_i:U_i\to \calU(N_{D_n})$ implementing an equivalence between the systems of transition functions of $(U_i,h_i)$ and $(U_i,h_i')$. The induced automorphism $T$ of $\E$ as a Hilbert $C_0(X)$-module has to intertwine the left action, that is, be an automorphism in the category of $C^*$-correspondences. 

We adopt the notation of Remark \ref{rem:transition functions vector and mn bundle}. The functions $\Ad(r_i)$, which map from $U_i$ to $\aut(M_n,D_n)$, implement an equivalence between the systems of transition functions of $(U_i,H_i)$ and $(U_i,H_i')$. This gives an automorphism of the principal $\aut(M_n,D_n)$-bundle $\mathcal{B}$. The induced automorphism of $\K(\E)\cong\Gamma_0(\mathcal{B})$ is the same as the one obtained from $T$. By \cite[Proposition 2.3]{LiRenault:2019} it sends $D$ to $D'$. This shows that $D$ and $D'$ are conjugate in the sense of Definition \ref{def:conjugacy of cartans}.

Now take two atlases $(U_i,h_i)$ and $(U_i,h_i')$ of $\V$ whose associated Cartan subalgebras $D$ and $D'$ are conjugate. This means that there is an automorphism $T$ of $\E$ such that the induced automorphism of $\K(\E)$ sends $D$ to $D'$. In particular this gives an automorphism of $\V$. We obtain  functions $r_i$ from $U_i$ to $\calU(M_n)$ which establish an equivalence between the systems of transition functions of $(U_i,h_i)$ and $(U_i,h_i')$. As in the first part of the proof, the functions $\Ad(r_i)$ define a principal $\aut(M_n)$-bundle automorphism $\Phi$ of $\mathcal{B}$. Again the induced automorphism of $\K(\E)$ is the same as the one obtained from $T$. It follows from \cite[Proposition 2.3]{LiRenault:2019} that $\Phi$ is in fact an isomorphism of principal $\aut(M_n,D_n)$ bundles. Hence the functions $\Ad(r_i)$ are $\aut(M_n,D_n)$ valued, which implies that the $r_i$ are $\calU(N_{D_n})$-valued. We obtain that $(U_i,h_i)$ and $(U_i,h_i')$ are equivalent in the sense of Definition \ref{def:equivalence diagonalizing atlases}. This finishes the proof.
\end{proof}
\subsection{Cartan subalgebras and correspondences}\label{sect:cartans and corrs}
Let $(Z,r,s,\calL)$ be a proper twisted topological $Y$-$X$-correspondence associated to a $C^*$-correspondence $\E$. Recall from Definition \ref{def:twisted local homeo module} that elements of $\E$ can be regarded as sections of $\calL$. Take a function $f\in C_0(Z)$. Then $f$ acts on $\E$ by pointwise multiplication, or equivalently by right multiplication in the $C_0(Z)$-module $\Gamma_0(\calL)$. Explicitly, we can define a map $\Theta_f:\E\to \E$ by $\Theta_f\xi(z)=f(z)\xi(z)$ for all $\xi\in\E$ and $z\in Z$.
\begin{proposition}\label{prop:twtg to cartan}
For every $f\in C_0(Z)$, the operator $\Theta_f$ is linear, adjointable, and compact. The set $\{\Theta_f:f\in C_0(Z)\}$ is a Cartan subalgebra of $\K(\E)$ containing the image of $\varphi$. 
\end{proposition}
\begin{proof}
    It is easy to see that $\Theta_f$ is $C_0(X)$-linear. Furthermore we have
\begin{align*}
    \langle\Theta_f\xi,\eta\rangle(x)&=\sum_{z\in s^{-1}(x)}\langle\Theta_f\xi,\eta\rangle_\calL(z)=\sum_{z\in s^{-1}(x)}\langle\xi,\eta\rangle_\calL(z)\bar{f}(z)\\&=\sum_{z\in s^{-1}(x)}\langle\xi,\Theta_{\bar{f}}\eta\rangle_\calL(z)=\langle\xi,\Theta_{\bar{f}}\eta\rangle(x)
\end{align*}
for all $\xi,\eta\in\E$. This implies $\Theta_f^*=\Theta_{\bar{f}}$, in particular $\Theta_f$ is adjointable. It is compact because $f$ vanishes at infinity. It is clear that $\Theta_f\Theta_g=\Theta_{fg}$, and hence $D\coloneqq \{\Theta_f:f\in C_0(Z)\}$ is commutative. 

To show that $D$ is a Cartan subalgebra, we employ the theory of the previous sections. Let $\mathcal{P}$ be the principal $\calU(N_{D_n})$-bundle associated to $(Z,r,s,\calL)$. Take an atlas of $\mathcal{P}$. We obtain sections $\xi_{ik}$ in $\Gamma_0(\mathcal{P}\times\C^n)$ and sections $\tilde{\xi}_{ik}$ in $\Gamma_0(Z,r,s,\calL)$ as in Section \ref{sect:atlas from corr}. Then there is an isomorphism $\Phi$ from $\Gamma_0(\mathcal{P}\times\C^n)$ to $\Gamma_0(Z,r,s,\calL)$ which, according to Lemma \ref{lem:canonical sections}, sends $\xi_{ik}$ to $\tilde{\xi}_{ik}$. 

We know from Theorems \ref{thm:corr and atlas one to one} and \ref{thm:cartan and atlas} that $\K(\Gamma_0(\mathcal{P}\times\C^n))$ has a Cartan subalgebra $D'$. By its construction, $D'$ is the unique maximal abelian subalgebra of $\K(\Gamma_0(\mathcal{P}\times\C^n))$ such that $\xi_{ik}(x)$ is an eigenvector of $d(x)$ for all $d\in D'$ and $x\in X$. The same holds for $D$ and $\tilde{\xi}_{ik}$. Thus we see that the isomorphism between the algebras of compact operators induced by $\Phi$ sends $D'$ to $D$. In particular, $D$ is a Cartan subalgebra of $\K(\E)$.

By definition of the left action we have $\varphi(f)=\Theta_{f\circ r}$ for all $f\in C_0(Y)$. The range map $r$ is proper, so $f\circ r$ lies in $C_0(Z)$. Thus the image of $\varphi_\E$ is contained in $D$. This finishes the proof. 
\end{proof}
The other direction, namely showing that a Cartan subalgebra in $\K(\E)$ containing the image of $\varphi$ gives rise to a twisted topological correspondence associated to $\E$, requires more effort. 

Let $Z$ and $X$ be topological spaces. A map $s$ from $Z$ to $X$ is a \emph{branched covering map} if it is a covering map everywhere except on a nowhere dense set called the \emph{branch set}. More concretely, this requires that there exists a nowhere dense set $B\subset X$ such that any point in $s(Z)\backslash B$ has a neighborhood evenly covered by $s$. 
\begin{lemma}\label{lem:abelian C_0(X) algebra branched covering}
    Let $\E$ be a full $\sigma$-finitely generated and $\sigma$-projective Hilbert $C_0(X)$-module of rank $n$. Let $D$ be an abelian subalgebra of $\K(\E)$ containing $C_0(X)$. Then the Gelfand dual $\iota^*:\widehat{D}\to X$ of the inclusion $\iota:C_0(X)\hookrightarrow D$ is a branched covering. For every point $x\in X$ the fiber $(\iota^*)^{-1}(x)$ is finite. The branching set of $\iota^*$ is closed. 
\end{lemma}
\begin{proof}
Note that since $C_0(X)$ contains an approximate unit for $\K(\E)$ and hence for $D$, the inclusion $\iota$ is nondegenerate. Thus we can indeed take the Gelfand dual $\iota^*$, and obtain a continuous proper map from $\widehat{D}$ to $X$. 

We fix some notation. Write $n$ for the mapping from $X$ to $\N$ that sends $x\in X$ to the dimension $n(x)$ of the fiber $D_x$ over a point $x\in X$. The set of natural numbers greater or equal to $k\in\N$ is denoted by $\N_k$. We write $B$ for the subset of points $x$  such that $n$ is not constant on any neighborhood of $x$. That is, $B$ consists of those points at which $n_x$ jumps between two values.  

Take $x\in X$. Since $D_x$ is an abelian $n(x)$-dimensional $C^*$-algebra, it must be isomorphic to $D_{n(x)}$. In particular the fiber $(\iota^*)^{-1}(x)$ has exactly $n(x)$ elements. By \cite[Theorem 3.1]{Fell:1961} there exists a neighborhood $U$ of $x$ and a mapping $\phi$ from $D_{n(x)}$ to $D$ such that $\phi(d)(x)=d$ for all $d\in D_{n(x)}$, and such that for each $y\in U$ the map $d\mapsto \phi(d)(y)$ is a $\ast$-isomorphism from $D_{n(x)}$ onto a finite-dimensional subalgebra of $D_y$. We obtain an injective $C_0(U)$-linear $\ast$-homomorphism from $C_0(U)\otimes D_{n(x)}$ to $D$ by sending $f\otimes d$ to $\iota(f)\phi(d)$. In particular $n^{-1}(\N_k)$ is open for every $k\in\N$. Hence the boundary $\partial n^{-1}(\N_k)$ is nowhere dense. 

If the fibers of $\E$ have rank at most $N$, then
\[B=\bigcup_{k=0}^N\partial n^{-1}(\N_k)\]
and therefore $B$ is nowhere dense. It is also closed, since it is a finite union of closed sets. 

Take $x$ in the complement of $B$. By the above there exists an open neighborhood $U$ of $x$ and an injective $C_0(U)$-linear $\ast$-homomorphism from $C_0(U)\otimes D_{n(x)}$ to $D$. We can assume that $U$ is contained in $X\backslash B$, and that $n$ is constant on $U$. Hence the map is in fact a $\ast$-isomorphism from $C_0(U)\otimes D_{n(x)}$ to $D_U$, because it induces an isomorphism in each fiber. Under this isomorphism, the embedding $\iota_U$ of $C_0(U)$ into $D_U$ corresponds to the embedding $f\mapsto f\otimes I_{n(x)}$ of $C_0(U)$ into $C_0(U)\otimes D_{n(x)}$. Here $I_{n(x)}$ is the identity matrix. The Gelfand dual $(\iota_U)^*=(\iota^*)_U$ of $\iota_U$ is the trivial $n(x)$-sheeted covering map from $\bigsqcup^{n(x)} U$ to $U$. Hence any $x$ from the complement of the branch set $B$ has a neighborhood that is evenly covered by $\iota^*$. That is, $\iota^*$ is a branched covering. 
\end{proof}
To prove the main result of this section we will only use Lemma \ref{lem:abelian C_0(X) algebra branched covering} in the case that the branching set $B$ is empty, and we are dealing with an actual covering. However, the more general case will be very useful in Section \ref{sect:reconstructing corr from left action}.
\begin{remark}\label{rem:projections on covering}
Let $\E$ be a full $\sigma$-finitely generated and $\sigma$-projective Hilbert $C_0(X)$-module of rank $n$. Let $D$ be an abelian subalgebra of $\K(\E)$ containing $C_0(X)$. By Lemma \ref{lem:abelian C_0(X) algebra branched covering} we obtain a branched covering $s:\widehat{D}\to X$. Let $B$ be its branching set. Take a point $x\in X\backslash B$ and let $K$ be an evenly covered compact neighborhood of $x$ contained in $X\backslash B$. We write $m$ for the constant cardinality of the fibers $s^{-1}(y)$ for $y\in K$. Then $s^{-1}(K)$ is homeomorphic to a disjoint union of $m$ copies of $K$. Hence we can find elements $p_1,\dots,p_m$ of $D$ such that $p_1(y),\dots,p_m(y)$ are pairwise orthogonal projections in $\K(\E)_y\cong M_n$ summing up to the identity, for every $y\in K$. If $z$ lies in $s^{-1}(x)$ then there exists exactly one $i\in\{1,\dots,m\}$ such that $p_i(x)$ is nonzero. Write $p_z\coloneqq p_i(x)\in M_n$. If $p_z$ and $\tilde{p}_z$ are two projections obtained in this way by choosing different evenly covered neighborhoods of $x$, then $p_z=\tilde{p}_z$. It therefore makes sense to speak of the projection $p_z$ associated to an element $z\in s^{-1}(x)$.   
\end{remark}

Assume that we are in the setting of Remark \ref{rem:projections on covering}. We again denote the Gelfand dual of the inclusion $C_0(X)\hookrightarrow D$, which is a branched covering by Lemma \ref{lem:abelian C_0(X) algebra branched covering}, by $s$. We write $\E^D$ for $\E$ regarded as a $D$-module. 

Let $K$ be a compact set contained in $X\backslash B$. We can assume that the fibers over points in $K$ have constant dimension $m\in\N$, otherwise we partition $K$ into a finite number of clopen sets over which the fibers have constant dimension (compare Remark \ref{rem:constant rank}). Then the restriction $s_K:s^{-1}(K)\to K$ is an $m$-sheeted covering map. For the following lemma note that $D_K$ is isomorphic to  $C(s^{-1}(K))$. 

\begin{lemma}\label{lem:structure of module from abelian subalgebra}
Let $K$ be a compact set contained in $X\backslash B$. Then $\E^D_K$ is a finitely generated projective $D_K$-module. The rank of $\E^D_K$ over a point $z\in s^{-1}(K)$ is equal to the rank of the projection $p_z$ from Remark \ref{rem:projections on covering}.
\end{lemma}
\begin{proof}
We can assume without loss of generality that the fibers over points in $K$ have constant dimension $m\in\N$. 

Take a finite collection of compact sets $\{K_i\}_{i\in I}$ which covers $K$, and such that all $K_i$ are evenly covered by $s$. As shown in Remark \ref{rem:projections on covering}, this implies that for every $i\in I$ there exist elements $p_{1i},\dots,p_{mi}$ of $D$ such that $p_{1i}(x),\dots,p_{mi}(x)$ are pairwise orthogonal projections in $\K(\E)_x\cong M_n$ summing up to the identity, for every $x\in K_i$. 

Since $\E$ is a $\sigma$-finitely generated and $\sigma$-projective $C_0(X)$-module, $\E_{K_i}$ is a finitely generated projective $C(K_i)$-module for every $i\in I$. Thus $p_{ki}\E_{K_i}$ is also finitely generated and projective, for every $k\in\{1,\dots,m\}$. We have $\E_{K_i}^D\cong\bigoplus_{k=1}^mp_{ki}\E_{K_i}$ and $D_{K_i}\cong\bigoplus_{k=1}^m C(K_i)$. Under these isomorphisms, right multiplication of an element $(f_1,\dots,f_m)$ of $D_{K_i}$ with an element $(\xi_1,\dots,\xi_m)$ of $\E_{K_i}^D$ is given by
\[(\xi_1,\dots,\xi_m)(f_1,\dots,f_m)=(\xi_1f_1,\dots,\xi_mf_m).\]
 Hence the bundle over $\bigsqcup_{k=1}^m K_i$ associated to $\E_{K_i}^D$, when restricted to the $k$-th copy of $K_i$, is the vector bundle associated to $p_{ki}\E_{K_i}$. Thus the bundle associated to $\E_{K_i}^D$ has finite rank and is locally trivial, which implies that $\E_{K_i}^D$ is finitely generated and projective.

Hence for every $i\in I$ the bundle over $\widehat{D}$ associated to $\E_K^D$ has finite rank and is locally trivial when restricted to $K_i$. This implies that it has finite rank and is locally trivial on all of $K$, and thus $\E_K^D$ is finitely generated projective. That the rank of $\E^D_K$ over a point $z\in s^{-1}(K)$ is equal to the rank of the projection $p_z$ also follows from this decomposition.
\end{proof}
We now want to turn $\E^D_K$ into a Hilbert $D_K$-module. Let $\{U_i\}_{i\in I}$ be a finite open cover of $K$, such that the closure of each $U_i$ is compact and evenly covered. 
Write $\beta$ for the metric on the vector bundle associated to $\E$ that gives the inner product. Then we define a map from $\E_K\times\E_K$ to $D_K$ by setting
\begin{equation}\label{eq:inner product}\langle\xi,\eta\rangle_D(z)=\beta_{s(z)}(p_z\xi(x),\eta(x))\end{equation}
for $\xi,\eta\in\E$ and $z\in s^{-1}(K)$, where $p_z$ is the projection in $D_{s(z)}$ constructed in Remark \ref{rem:projections on covering}.
\begin{lemma}
The map $\langle\cdot,\cdot\rangle_D$ is a $D_K$-valued inner product turning $\E^D_K$ into a Hilbert $D_K$-module.
\end{lemma}
\begin{proof}
It is easy to see that $\langle\cdot,\cdot\rangle_D$ fulfills the inner product axioms. Since $\E_K^D$ is finitely generated projective, completeness is automatic, see \cite{Waldmann:2021}.
\end{proof}

Now assume that $D$ is not just an abelian subalgebra of $\K(\E)$, but a Cartan subalgebra. In other words, we require that for all $x\in X$ the dimension of $D_x$ equals the dimension of $\E_x$. Then $D$ automatically contains $C_0(X)$, since it is maximal abelian. The branching set $B$ is empty, and $s:\widehat{D}\to X$ is an $n$-sheeted covering. We obtain from Lemma \ref{lem:structure of module from abelian subalgebra} that $\E^D$ is $\sigma$-finitely generated and $\sigma$-projective. We can use Equation (\ref{eq:inner product}) to define a $D$-valued inner product turning $\E^D$ into a Hilbert $D$-module. Furthermore, we have $\rank p_z=1$ for all $z\in\widehat{D}$, where $p_z$ is the projection from Remark \ref{rem:projections on covering}. Then Lemma \ref{lem:structure of module from abelian subalgebra} implies that the vector bundle over $\widehat{D}$ associated to $\E^D$ is a line bundle, which we call $\calL$. Thus $(Z,s,\calL)$ defines a twisted covering.

\begin{proposition}\label{prop:cartan in compacts iso of modules}
The Hilbert $C_0(X)$-modules $\E$ and $\Gamma_0(Z,s,\calL)$ are isomorphic. 
\end{proposition}
\begin{proof}
Algebraically, $\Gamma_0(Z,s,\calL)$ is the same as $\E^D$ with right multiplication restricted from $D$ to $C_0(X)$. Since $\E^D$ was just a notation for $\E$ regarded as a right $D$-module, we see that as $C_0(X)$-modules $\Gamma_0(Z,s,\calL)$ is the same as $\E$. We will not distinguish between elements of both modules. To see that the inner products coincide, take $x\in X$ and calculate
\begin{align*}
    \langle\xi,\eta\rangle_{\Gamma_0(Z,s,\calL)}(x)=\sum_{z\in s^{-1}(x)}\langle\xi,\eta\rangle_{\calL}(z)=\sum_{z\in s^{-1}(x)}\beta_{x}(p_z\xi(x),\eta(x))=\beta_{x}(\xi(x),\eta(x)).
\end{align*}
The last equality follows from the fact that the $p_z$ sum up to the identity. By definition $\beta_x(\xi(x),\eta(x))=\langle\xi,\eta\rangle_\E(x)$, and hence we are done.  
\end{proof}

Now we come to the main theorem of this section. Let, as in the previous sections, $\E$ be a proper nondegenerate $\sigma$-finitely generated and $\sigma$-projective $C^*$-correspondence from $C_0(Y)$ to $C_0(X)$ with left action $\varphi$.
\begin{theorem}\label{thm:corr to cartan}
There is a bijection between isomorphism classes of proper twisted $Y$-$X$-correspondences $(Z,r,s,\calL)$ associated to $\E$, and conjugacy classes of Cartan subalgebras of $\K(\E)$ containing the image of $\varphi$.
\end{theorem}
\begin{proof}
We assume once more without loss of generality that for all $x\in X$, the rank of $\E_x$ is equal to $n\in\N$.

Assume that $D$ is a Cartan subalgebra of $\K(\E)$ containing the image of $\varphi$. Then we obtain a twisted covering $(Z,s,\calL)$ from the discussion before Proposition \ref{prop:cartan in compacts iso of modules}. Here $s:\widehat{D}\to X$ is the Gelfand dual of the inclusion of $C_0(X)$ into $D$, and $\calL$ is the complex line bundle we obtain from regarding $\E$ as a $D$-module. Proposition \ref{prop:cartan in compacts iso of modules} shows that $\Gamma_0(Z,s,\calL)$ is isomorphic to $\E$. Using Proposition \ref{prop:twtg to cartan} to construct a Cartan in $\K(\Gamma_0(Z,s,\calL))$ gives back $D$. Hence the reconstruction of a twisted covering from a Cartan subalgebra is inverse to obtaining a Cartan from a twisted covering. 

Now we take the left action into account. Since $\varphi$ is nondegenerate and its image is contained in $D$, Gelfand duality yields a proper continuous map $r$ from $\widehat{D}$ to $Y$. Thus we obtain a proper twisted topological $Y$-$X$-covering $(Z,r,s,\calL)$ associated to $\E$.

Conversely, assume that $(Z,r,s,\calL)$ is a proper twisted topological $Y$-$X$-correspondence associated to $\E$. We have shown in Proposition \ref{prop:twtg to cartan} that the image of $\varphi$ is contained in the Cartan subalgebra. 

That the constructions preserve isomorphism classes follows directly from the definitions. 
\end{proof}

\subsection{Completing the diagram}

There are two last results completing the picture that we have drawn so far. The first one concerns the commutativity of Diagram (\ref{diag:triangle}). The second result tells us what happens if, instead of fixing the $C^*$-correspondence, we take two different $C^*$-correspondences and ask when the associated topological correspondences, atlases or Cartan subalgebras are isomorphic.

\begin{proposition}\label{prop:commutativity of diagram}
Diagram (\ref{diag:triangle}) commutes up to isomorphism. 
\end{proposition}
\begin{proof}
If $(Z,r,s,\calL)$ is a proper twisted topological $Y$-$X$-correspondence associated to $\E$, then we have two ways of obtaining a Cartan subalgebra. We can either use Theorem \ref{thm:corr to cartan}, or take a detour via Theorem \ref{thm:corr and atlas one to one} and \ref{thm:cartan and atlas}. Denote the first Cartan by $D$, and the second one by $D'$. 

We have shown in the proof of Proposition \ref{prop:twtg to cartan} that the isomorphism from $\K(\Gamma_0(\mathcal{P}\times\C^n))$ to $\K(\Gamma_0(Z,r,s,\calL))$ sends $D'$ to $D$. This concludes the proof.
\end{proof}

So far, we have only dealt with topological correspondences, atlases and Cartans associated to a fixed $C^*$-correspondence. This changes in the following proposition. Notice that Condition (c) did not appear before, while the other conditions were used previously. We include Condition (c) because we will use it in the proof of Proposition \ref{prop:local conjugacy}.
\begin{proposition}\label{prop:isos for different Cstar correspondences}
Let $\E$ and $\E'$ be $C^*$-correspondences from $C_0(Y)$ to $C_0(X)$ which have proper twisted $Y$-$X$-correspondences $(Z,r,s,\calL)$ and $(Z',r',s',\calL')$, respectively, associated to them. Let $(U_i,h_i)_{i\in I}$ and $(U_i,h_i')_{i\in I}$ be normalizing atlases diagonalizing the left actions that correspond to $(Z,r,s,\calL)$ and $(Z',r',s',\calL')$ as in Theorem \ref{thm:corr and atlas one to one}. Let $D$ and $D'$ be Cartan subalgebras associated to the $Y$-$X$-correspondences via Theorem \ref{thm:corr to cartan}. Then the following are equivalent:
\begin{itemize}
\item[(a)] The twisted correspondences $(Z,r,s,\calL)$ and $(Z',r',s',\calL')$ are isomorphic. 
\item[(b)] The atlases $(U_i,h_i)$ and $(U_i,h_i')$ are equivalent in the sense of Definition \ref{def:equivalence diagonalizing atlases}.
\item[(c)] There exists an isomorphism $\Phi:\E\to\E'$ of $C^*$-correspondences such that the following holds: For all $i\in I$, all $x\in U_i$ and all $k=1,\dots,n_x$ there exists an $l\in\N$ such that $\Phi(\xi_{ik})(x)=\lambda \xi_{il}'(x)$ for some $\lambda\in\T$, where the sections $\xi_{ik}$ and $\xi_{ik}'$ are defined from the two atlases as in Remark \ref{rem:frame}.
\item[(d)] There exists an isomorphism $\Phi:\E\to\E'$ of $C^*$-correspondences such that the induced isomorphism from $\K(\E)$ to $\K(\E')$ maps $D$ to $D'$. 
\end{itemize}
\end{proposition}
\begin{proof}
The proofs of the implications from (d) to (a) and from (a) to (b) are exactly the same as in Theorems \ref{thm:corr and atlas one to one} and \ref{thm:corr to cartan}.

We show that (b) implies (c). Recall from the beginning of Section \ref{sect:equiv of atlases} that if the functions implementing the equivalence between the atlases are called $r_i$, then the isomorphism from $\V$ to $\V'$ is defined by sending $v\in\V_x$ with $x\in U_i$ to $h_i'(x,r_i(x)^{-1}h_i^{-1}(x,v))$. Hence $\xi_{ik}(x)=h_i(x,\gamma_i(x)e_k)$ is send to $h_i'(x,\gamma_i(x)r_i(x)^{-1}e_k)$. Since $r_i(x)$ lies in $\calU(N_{D_{n_x}})$, there exist $l\in\{1,\dots,n_x\}$ and $\lambda\in\T$ such that $r_i^{-1}(x)e_k=\lambda e_l$. Thus the image of $\xi_{ik}(x)$ is $\lambda\xi_{il}(x)$. This establishes (c).

For the implication from (c) to (d) recall from Equation (\ref{eq:cartan in compacts}) that for all $i\in I$ we have $\mathcal{D}_{U_i}=H_i(U_i\times D_n)$, where $H_i$ is given by mapping $(x,a)\in U_i\times M_n$ to the operator $h_i(x,ah_i^{-1}(x,\cdot))$. Here we again assume that $\E$ has constant rank $n$. Hence for every $x\in U_i$ the fiber $D_x$ is the linear span of the operators $\{\xi_{ik}(x)\xi_{ik}(x)^*\}_{k=1}^{n}$. The same holds for $D'$. If $\Psi:\K(\E)\to\K(\E')$ is the isomorphism induced by $\Phi$, then 
\begin{equation}\label{eq:iso preserves cartan}\Psi_x(\xi_{ik}(x)\xi_{ik}^*(x))=\Phi(\xi_{ik})(x)\Phi(\xi_{ik})(x)^*=\xi_{il}'(x)\xi_{il}'(x)^*.\end{equation}
Since $\Phi$ preserves the inner product, $\Phi_x$ does as well. Hence the image of the set $\{\xi_{ik}(x)\}_{k=1}^n$ under $\Phi_x$ is the set $\{\xi'_{ik}(x)\}_{k=1}^n$. Thus every $l=1,\dots,n$ appears exactly once on the right hand side of Equation (\ref{eq:iso preserves cartan}). Therefore $\Psi_x$ maps $D_x$ to $(D')_x$. Since $i\in I$ and $x\in U_i$ were arbitrary, this shows that $\Phi$ maps $D$ to $D'$. We have established (d).
\end{proof}

\section{Applications and examples}\label{sec:applications}
\subsection{Twisted coverings over spheres}\label{sec:twisted coverings over spheres}
As a concrete example, we want to investigate which twisted coverings, and hence which twisted topological graphs and twisted topological correspondences, exist over the spheres $S^k$. This means that for any given $k\in\N$ we first need to characterize the covering spaces over $S^k$, and then complex line bundles on those covering spaces. 

For any connected space $X$, it is a classical result that $n$-sheeted covering spaces of $X$ are classified by equivalence classes of homomorphisms from the fundamental group $\pi_1(X)$ to $S_n$, see for example \cite{Hatcher:2002}. Here two homomorphisms $\rho$ and $\chi$ are called equivalent if there exists $g\in S_n$ such that $\rho=g\chi g^{-1}$. Since all spheres except $S^1$ are simply connected, we immediately see that only $S^1$ has nontrivial covering spaces. 

The fundamental group of $S^1$ is given by $\pi_1(S^1)\cong\Z$. Hence $n$-sheeted covering spaces of $S^1$ are classified by equivalence classes of homomorphisms from $\Z$ to $S_n$. Such a homomorphism is completely determined by specifying the image of the generator $1\in\Z$. Taking the equivalence relation into account, we see that there is exactly one equivalence class of homomorphisms for each conjugacy class of $S_n$. 

Vector bundles over the spheres are described by the well-known clutching construction. According to this construction, isomorphism classes of complex rank $n$ vector bundles over $S^k$ correspond bijectively to homotopy classes of continuous maps from $S^{k-1}$ to $Gl(n,\C)$, that is, to elements of $\pi_{k-1}(Gl(n,\C))$. We are interested in complex line bundles, so we have to consider $\pi_{k-1}(\C^\times)$. 

It follows that there are no nontrivial complex line bundles over $S^1$, and more generally, no nontrivial complex vector bundles. Note that this gives a simple explanation of the example in \cite[Section 5]{frausino2023}. There the authors considered the two different two-fold covering spaces of $S^1$, and showed that the associated vector bundles are trivial. 

Since all covering spaces of $S^1$ are disjoint copies of $S^1$, we obtain that any twisted covering over $S^1$ has trivial twist. To summarize, for every $n\in\N$ there is only the trivial complex vector bundle of rank $n$ over $S^1$. This trivial vector bundle has exactly $p(n)$ associated coverings, where $p(n)$ is the number of conjugacy classes of $S_n$. None of these coverings admits a nontrivial twist. 

For $S^2$ we have seen above that there are no nontrivial covering spaces. However, we have $\pi_1(\C^\times)\cong\Z$ and hence there are nontrivial line bundles. We write $H$ for the nontrivial line bundle whose first Chern class $c_1(H)$ is the generator $1\in\Z\cong H^2(S^2,\Z)$. One can show that $c_1(\V_1\oplus \V_2)=c_1(\V_1)+c_1(\V_2)$ holds for all vector bundles $\V_1$ and $\V_2$ over $S^2$. Furthermore, we always have $c_1(\calL_1\otimes\calL_2)\cong c_1(\calL_1)+c_1(\calL_2)$ for line bundles $\calL_1$ and $\calL_2$. Using these two identities we see that the rank $n$ bundle \[H_{l_1,\dots,l_n}\coloneqq \bigoplus_{i=1}^nH^{l_i}\] for $l_1,\dots,l_n\in\Z$ has first Chern class $l_1+\dots+l_n$. Here we use the conventions that $H^{-l}\coloneqq (H^*)^l$, where $H^*$ is the adjoint bundle, and that $H^0$ is the trivial line bundle. It is a well-known fact that the first Chern class together with the rank is a complete invariant for vector bundles on CW-complexes of dimension $\leq 4$, see for example \cite{Rosenberg:2024}. Hence every complex rank $n$ vector bundle $\V$ over $S^2$ is isomorphic to $H_{l_1,\dots,l_n}$, where $l_1,\dots,l_n$ have to be chosen such that $l_1+\dots+l_n=c_1(\V)$. We obtain that every complex vector bundle over $S^2$ has countably many pairwise inequivalent twisted coverings associated to it, whose underlying covering spaces are trivial. 

We now turn to $S^k$ for $k\geq 3$. Each of those spheres is simply connected, and hence has no nontrivial covering spaces. They also do not support any nontrivial line bundles, since $\pi_{k-1}(\C^\times)$ is trivial for $k\geq 3$. However, they do support nontrivial complex vector bundles of rank greater than one. We thus see that there are vector bundles on $S^k$ for $k\geq 3$ that do not have associated twisted coverings. In fact, only the trivial bundle of each rank has a twisted covering associated to it. It is the trivial covering with trivial twist.  

Using these results and Theorem \ref{thm:corr to cartan}, we can describe Cartan subalgebras in $n$-homogeneous $C^*$-algebras over spheres. It follows from \cite[Theorem 4.85]{raeburnwilliams:1998} that such an $n$-homogeneous algebra is of the form $\K(\Gamma_0(\V))$, for $\V$ a vector bundle over $S^k$, if and only if its Dixmier--Douady class vanishes. By \cite[Hooptedoodle 4.91]{raeburnwilliams:1998}, this Dixmier--Douady class is an $n$-torsion element of $H^3(S^k,\Z)$. This group is nonzero only for $k=3$, in which case it is given by the integers. Thus the Dixmier--Douady class of a $n$-homogeneous $C^*$-algebra over a sphere always vanishes. 

This reproduces the results regarding Cartan subalgebras of $n$-homogeneous $C^*$-algebras over spheres from \cite[Section 2.2]{LiRenault:2019}. For $k=4$ we recover the example from the appendix of \cite{Kumjian:1985} of a homogeneous $C^*$-algebra with spectrum $S^4$ which does not have a Cartan subalgebra.

\subsection{Covering space with nontrivial associated vector bundle}\label{sect:covering with nontrivial vector bundle}

As discussed in \cite{frausino2023}, it does not seem to be known whether there exists a covering space whose associated vector bundle is nontrivial. We have seen in Section \ref{sec:twisted coverings over spheres} that there does not exist one if the base space is a sphere. Using the results of Section \ref{sect:twisted coverings and principal bundles} we obtain that the question is equivalent to the following: Does there exist a principal $S_n$-bundle over some space, and for some $n\in\N$, whose associated fiber bundle under the action of $S_n$ on $\C^n$ via permutation matrices is nontrivial? 

 The following example was provided by ChatGPT \cite{openai2025chatgpt}: Consider the real projective plane $\R P^2$. It has fundamental group $\Z_2$ and universal covering $p:S^2\to \R P^2$, which is two-sheeted. We obtain a principal bundle $\mathcal{P}=(S^2,p,\R P^2)$ with structure group $S_2=\Z_2$. 

Write $\pi$ for the representation of $S_2$ on $\C^2$ by permutation matrices. This representation splits into a direct sum of the two irreducible representations of $S_2$: We have $\pi=\pi_+\oplus\pi_-$, where $\pi_+$ is the trivial representation on $\C$, and $\pi_-$ is the sign representation that maps $-1\in\Z_2=\{1,-1\}$ to the matrix $(-1)$. Thus the vector bundle splits into a direct sum, 
\[\mathcal{P}\times_\pi\C^2=(\mathcal{P}\times_{\pi_+}\C)\oplus(\mathcal{P}\times_{\pi_-}\C).\]
Clearly $\mathcal{P}\times_{\pi_+}\C$ is the trivial line bundle over $\R P^2$. The bundle $\mathcal{P}\times_{\pi_-}\C$ has non-trivial total Chern class: Write $(\mathcal{P}\times_{\pi_-}\C)_\R$ for $\mathcal{P}\times_{\pi_-}\C$ regarded as a real vector bundle (which has rank two). It is the direct sum of two copies of $\mathcal{P}\times_{\pi_-}\R$, where by slight abuse of notation we use $\pi_-$ to denote the sign representation on $\R$. It is easy to show that $\mathcal{P}\times_{\pi_-}\R$ is in fact the tautological line bundle on $\R P^2$. It follows that the second Stiefel-Whitney class $w_2$ of $(\mathcal{P}\times_{\pi_-}\C)_\R$ is the generator of $H^2(\R P^2,\Z_2)\cong\Z_2$. 

By \cite[Proposition 3.8]{Hatcher:2003}, $w_2((S^2\times_{\pi_-}\C)_\R)$ is the image of the first Chern class $c_1(S^2\times_{\pi_-}\C)$ under the coefficient homomorphism from $H^2(\R P^2,\Z)$ to $H^2(\R P^2,\Z_2)$. Since $H^2(\R P^2,\Z)=H^2(\R P^2,\Z_2)=\Z_2$ this coefficient homomorphism is in fact an isomorphism of groups. Thus $c_1(S^2\times_{\pi_-}\C)$ is nontrivial, which concludes the proof. 

By additivity of the total Chern class, we obtain that $c(\mathcal{P}\times_\pi\C^2)\neq 1$. In particular, $\mathcal{P}\times_\pi\C^2$ is nontrivial. 

Before ChatGPT, Georg Jakob had suggested to the author to look at the two-sheeted covering of $\mathrm{SO}(n)$ by $\mathrm{Spin}(n)$. The author has not pursued this further, but it seems to be in a similar spirit and might provide another example. 

\subsection{Reconstructing topological correspondences from the left action}\label{sect:reconstructing corr from left action}
In \cite{frausino2023} the authors provided an example of two different coverings whose associated vector bundles are isomorphic. In Section \ref{sec:twisted coverings over spheres} we have seen more such examples. Thus when passing from a twisted covering to its Hilbert module, one in general loses information.

In these cases it is sometimes possible to reconstruct the covering from the left action. The basis for this is the observation from the proof of Proposition \ref{prop:correspondence to atlas} that if $(Z,r,s,\calL)$ is a twisted $Y$-$X$-correspondence, then the eigenvalues of $\varphi(f)_x$ are given by $\{f(r(y)):y\in s^{-1}(x)\}$, where $\varphi$ is the left action of the $C^*$-correspondence $\Gamma_0(Z,r,s,\calL)$. We will now see an example of how this observation can be used to reconstruct a topological graph from its associated $C^*$-correspondence. 
\begin{example}\label{ex:reconstructing trivial covering of the circle}
Take $Z=S^1\times\{1,2\}$ and $X=Y=S^1$ and define $s:Z\to X$ as $s(x,i)=x$ for all $x\in S^1$ and $i\in\{1,2\}$. This is the trivial 2-sheeted covering of $S^1$. Choose two distinct points $x_1$ and $x_2$ in $S^1$ and define $r:Z\to X$ as $r(x,i)=x_i$ for $i\in\{1,2\}$. We obtain a topological graph $(Z,r,s)$. The associated $C^*$-correspondence is given by the trivial Hilbert $C(S^1)$-module of rank two, namely $C(S^1)\oplus C(S^1)\cong C(S^1)\otimes \C^2$, with the left action 
\[\varphi(f)=1_{C(S^1)}\otimes\begin{pmatrix} f(x_1)&0\\0&f(x_2)\end{pmatrix}\]
where we have chosen an appropriate basis on the fibers. It is possible to recover the topological graph from this left action: The eigenvalues of $\varphi(f)_x$ are $f(x_1)$ and $f(x_2)$ for all $x\in S^1$ and $f\in C(S^1)$. This implies that the image of the range map consists of the two distinct points $x_1$ and $x_2$. The Hilbert module has rank two, and there are only two 2-sheeted coverings of $S^1$. One is the trivial covering, and the other one is given by mapping $z\in S^1$ to $z^2\in S^1$, and is connected. Since continuous functions map connected sets to connected sets, the image of the range map can only consist of two distinct points if the covering is trivial. 
\end{example}
We will exploit this observation more systematically, and in a slightly different framework, using the results of Section \ref{sect:perspective of the compacts}. There we have seen that if $\E$ is a proper nondegenerate $\sigma$-finitely generated and $\sigma$-projective $C^*$-correspondence from $C_0(Y)$ to $C_0(X)$, then specifying a twisted topological correspondence $(Z,r,s,\calL)$ associated to $\E$ is the same thing as specifying a Cartan subalgebra in $\K(\E)$ containing the image of $\varphi$.

 Assume that $\E$ is full as a right Hilbert $C_0(X)$-module. We will make this assumption throughout this section. It is convenient, because it ensures that right multiplication with any function in $C_0(X)$ gives a compact operator. As explained in Remark \ref{rem:constant rank}, it is not a serious restriction. We write $\iota$ for the map from $C_0(X)$ to $\K(\E)$ given by regarding functions as compact operators acting by right multiplication.

 Say we are given a Cartan subalgebra $C$ of $\K(\E)$ which contains the image of $\varphi$. Let $D$ be an abelian subalgebra of $C$ which contains the images of $\iota$ and $\varphi$. Write $s_D$ and $r_D$ for the Gelfand duals of $\iota$ and $\varphi$, respectively, regarded as maps from $C_0(X)$ to $D$. By Lemma \ref{lem:abelian C_0(X) algebra branched covering} we know that $s_D:\widehat{D}\to X$ is a branched covering. It follows from the functoriality of Gelfand duality that the diagram 
\[\begin{tikzcd}
	& X \\
	{\widehat{D}} && {\widehat{C}} \\
	& X
	\arrow["{s_D}", from=2-1, to=1-2]
	\arrow["{r_D}"', from=2-1, to=3-2]
	\arrow["s"', from=2-3, to=1-2]
	\arrow["{\widehat{j}}"', from=2-3, to=2-1]
	\arrow["r"', from=2-3, to=3-2]
\end{tikzcd}\]
commutes, where $\widehat{j}$ is the Gelfand dual of the inclusion $j$ of $D$ into $C$. In other words, the topological correspondence $(\widehat{C},r,s)$ has to factor through the tripel $(\widehat{D},r_D,s_D)$. Note however that this triple is not a topological correspondence in general, since $s_D$ is not a local homeomorphism unless the rank of $D$ is locally constant.

If we are not given the Cartan subalgebra, but just the $C^*$-correspondence $\E$, then we can look for abelian subalgebras of $\K(\E)$ which contain the images of $\iota$ and $\varphi$, and which have to be contained in \emph{any} Cartan subalgebra of $\K(\E)$ containing $\im\varphi$. The most obvious choice is the $C^*$-algebra generated by the images of $\iota$ and $\varphi$, which we write as $C^*(\im\iota,\im\varphi)$. It is abelian and contains both $\im\iota$ and $\im\varphi$. Furthermore it has to be contained in any subalgebra of $\K(\E)$ containing $\im\iota$ and $\im\varphi$. 

\begin{example} \label{ex:reconstructing trivial covering of the circle 2}
We revisit Example \ref{ex:reconstructing trivial covering of the circle}. We have 
\[C^*(\im\iota,\im\varphi)\cong C(S^1)\otimes D_2\]
whose Gelfand dual is $S^1\times\{1,2\}$. The inclusion $\iota$ is given by mapping $f\in C(S^1)$ to $f\otimes 1_{M_2}$. Its Gelfand dual $s=\widehat{\iota}:S^1\times\{1,2\}\to S^1$ is the trivial 2-sheeted covering. The range map, which is the Gelfand dual of $\varphi:C(S^1)\to C^*(\im\iota,\im\varphi)$, is given by mapping $(x,i)$ to $x_i$, for $i\in\{1,2\}$. Hence this recovers the topological graph we started with in Example \ref{ex:reconstructing trivial covering of the circle}.
\end{example}

As another application of this approach we show that, depending on the spaces $X$ and $Y$, there might exist $C^*$-correspondences from $C_0(Y)$ to $C_0(X)$ that are trivial as right Hilbert $C_0(X)$-modules, but do not come from a twisted $Y$-$X$-correspondence.   

\begin{example}
    Take $X=Y=S^3$. We have seen in Section \ref{sec:twisted coverings over spheres} that there exists a vector bundle $\V$ over $S^3$ with no twisted covering associated to it. Let $\mathcal{W}$ be a vector bundle such that $\V\oplus\mathcal{W}$ is isomorphic to the trivial bundle $S^3\times\C^n$ for some $n\in\N$. We define a left action $\varphi(f)(\xi,\eta)=(f(x_1)\xi,f(x_2)\eta)$ for distinct points $x_1$ and $x_2$ in $S^3$, where $(\xi,\eta)$ is an element of $\Gamma(S^3\times\C^n)$ with $\xi\in\Gamma(\V)$ and $\eta\in\Gamma(\mathcal{W})$. 

    We have $C^*(\im\iota,\im\varphi)\cong C(S^3)\otimes D_2$. If there was a twisted topological graph associated to $(\Gamma(S^3\times\C^n),\varphi)$, then it would have to factor through the trivial $2$-sheeted covering of $S^3$. However, this would give a twisted covering associated to $\Gamma(\V)$, which we know is not possible. 
\end{example}
It is clear that the larger the abelian subalgebra $D$ is, the more information about the topological correspondence we can generally expect to recover. To find a larger algebra than $C^*(\im\iota,\im\varphi)$ we need a lemma first.
\begin{lemma}
Let $A$ be a $C^*$-algebra and $C\subset A$ a maximal abelian subalgebra. If $D$ is an abelian subalgebra of $A$ contained in $C$, then the double commutant $D''$ relative to $A$ is contained in $C$ as well.
\end{lemma}
\begin{proof}
Since $C$ is abelian and $D$ is contained in $C$, any element of $C$ commutes with all elements of $D$. Hence $C$ is contained in $D'$. This implies that $D''$ is contained in $C'$. But since $C$ is maximal abelian inside $A$, the commutant of $C$ inside of $A$ is $C$ itself. This concludes the proof. 
\end{proof}
Hence the double commutant of $C^*(\im\iota,\im\varphi)$ inside of $\K(\E)$ is another abelian subalgebra that is contained in any Cartan subalgebra of $\K(\E)$ which contains the image of $\varphi$. Thus we can use it to obtain information about twisted topological correspondences associated to $\E$.  We want to understand this double commutant better.

\begin{lemma}\label{lem:double commutant}
Let $\E$ be a full $\sigma$-finitely generated and $\sigma$-projective Hilbert $C_0(X)$-module of rank $n$. Let $D$ be an abelian subalgebra of $\K(\E)$ containing $C_0(X)$. Write $B$ for the closed branching set of the branched covering associated to $D$ as in Lemma \ref{lem:abelian C_0(X) algebra branched covering} Then 
\[D''=\{a\in \K(\E):a(x)\in D_x\quad\text{for all }x\in X\backslash B\}.\]
\end{lemma}
\begin{proof}
Write $\tilde{D}$ for the algebra on the right hand side of the above equality. Take $d\in\tilde{D}$ and $a\in D'$. We have $d(x)a(x)=a(x)d(x)$ for all $x\in X\backslash B$. Since $X\backslash B$ is dense in $X$, and using continuity, we see that the equality $d(x)a(x)=a(x)d(x)$ in fact holds for all $x\in X$. In other words, $d$ lies in $D''$. Hence $\tilde{D}$ is contained in $D''$. 

We show the other inclusion. By the proof of Lemma \ref{lem:abelian C_0(X) algebra branched covering}, any point $x$ in $X\backslash B$ has an open neighborhood $U$ such that there exists a base-preserving isomorphism from $C_0(U)\times D_m$ to $D_U$, where $m$ is the dimension of $D_x$. In particular we can find elements $p_1,\dots,p_n$ of $D$ such that for each $y\in U$, the $p_i(y)$ are pairwise orthogonal projections summing up to the identity of $D_y$. The same strategy as in the proof of \cite[Lemma 2.1]{LiRenault:2019}, plus an application of the Gram--Schmidt theorem, yields a base-preserving isomorphism from $\K(\E)_U$ to $C_0(U)\otimes M_n$ sending $D_U$ to $C_0(U)\otimes D_{k_1,\dots,k_m}$. Here $D_{k_1,\dots,k_m}$ is the subalgebra of $M_n$ given by $\C 1_{k_1}\oplus \dots\oplus \C 1_{k_m}$, and $k_1,\dots,k_m$ are natural numbers summing up to $n$. Thus $(D'')_x$ is equal to $D_x$, which concludes the proof. 
\end{proof}
The next example shows that considering the double commutant of $C^*(\im\iota,\im\varphi)$ is sometimes necessary to reconstruct the twisted topological correspondence.
\begin{example}\label{ex:different range maps}
Take $X=Y=[-1,1]$, $Z=[-1,1]\times\{1,2\}$ and let $s$ be the trivial covering $s(x,i)=x$. We define two different range maps,
\[r_1(x,i)=\begin{cases}x\quad&\text{for }i=1,\\-x\quad&\text{for }i=2, \end{cases}\quad\text{and}\quad r_2(x,i)=\begin{cases}|x|\quad&\text{for }i=1,\\-|x|\quad&\text{for }i=2. \end{cases}\]
Denote the left actions of the $C^*$-correspondences $\Gamma_0(Z,r_1,s)$ and $\Gamma_0(Z,r_2,s)$ by $\varphi_1$ and $\varphi_2$, respectively. Then we have
\[C^*(\im\iota,\im\varphi_1)=C^*(\im\iota,\im\varphi_2)=\{f\in C(X,D_2):f(0)\in \C 1_{M_2}\}.\]
The Gelfand dual of this is the wedge sum $[-1,1]\vee[-1,1]$ with the two copies of the interval glued together at 0. The Gelfand dual of the inclusion $\iota$ is in both cases the obvious surjection from $[-1,1]\vee[-1,1]$ to $[-1,1]$. There is a base-preserving homeomorphism on $[-1,1]\vee[-1,1]$ intertwining the two range maps. Thus we see that not only can we not reconstruct the graphs from $C^*(\im\iota,\im\varphi_i)$, we cannot even distinguish them from each other. However, this is possible using the double commutants. We have $C^*(\im\iota,\im\varphi_1)''=C^*(\im\iota,\im\varphi_2)''=C(X,D_2)$, which is immediately clear from Lemma \ref{lem:double commutant}. Hence we can reconstruct the graphs in the same way as in Example \ref{ex:reconstructing trivial covering of the circle 2}.
\end{example}

Examples \ref{ex:reconstructing trivial covering of the circle 2} and \ref{ex:different range maps} show that it is sometimes possible to reconstruct the topological correspondence only from the subalgebras $C^*(\im\iota,\im\varphi)$ or $C^*(\im\iota,\im\varphi)''$ together with $\iota$ and $\varphi$. We want to investigate how much information these subalgebras contain in general. 

\begin{definition}\label{def:iso of subalgebras}
 Let $D_\E$ and $D_\F$ be subalgebras of $\K(\E)$ and $\K(\F)$ containing the images of $\iota_\E$ and $\varphi_\E$, and of $\iota_\F$ and $\varphi_\F$, respectively. Then we write $D_\E\cong_{\iota,\varphi}D_\F$ if there exists a $\ast$-isomorphism $\phi$ from $D_\E$ to $D_\F$ such that $\iota_\F=\phi\circ\iota_\E$ and $\varphi_\F=\phi\circ\varphi_E$. 
\end{definition}

\begin{lemma}\label{lem:range map gives iso of different coverings}
Take locally compact Hausdorff spaces $X$, $Z_1$ and $Z_2$. Let $s_i:Z_i\to X$, $i\in\{1,2\}$ be two proper continuous maps such that for all $x\in X$ the set $s_i^{-1}(x)$ is finite. Assume that there exist maps $r_i:Z_i\to Y$ into a Hausdorff space $Y$ such that for all $x\in X$, the sets $r_1(s_1^{-1}(x))$ and $r_2(s_2^{-1}(x))$ are equal.  Assume also that for all $x\in X$ and $i\in\{1,2\}$ we have $r_i(y)\neq r_i(z)$ for all distinct $y$ and $z$ in $s_i^{-1}(x)$. Then there exists a unique homeomorphism $\phi$ from $Z_1$ to $Z_2$ such that $s_2\circ\phi=s_1$ and $r_2\circ\phi=r_1$. 
\end{lemma}
\begin{proof}
Take $x\in X$ and $y\in s_1^{-1}(x)$. By assumption there exists exactly one $\tilde{y}$ in $s_2^{-1}(x)$ such that $r_1(y)=r_2(\tilde{y})$. Hence sending $y$ to $\tilde{y}$ yields a bijective map $\phi$ from $Z_1$ to $Z_2$. Any map $\psi$ such that $s_2\circ\psi=s_1$ and $r_2\circ\psi=r_1$ must send $y$ to $\tilde{y}$, so $\phi$ is unique. 

To show that $\phi$ is continuous, take any open neighborhood $\tilde{U}$ of $\tilde{y}$. Since $r_1$ and $r_2$ are continuous, and since $s_2^{-1}(x)$ is a finite set, there exist open neighborhoods $U$ of $y$ and $U_{\tilde{z}}$ of $\tilde{z}$ for all $\tilde{z}$ in $s_2^{-1}(x)\backslash\{\tilde{y}\}$ such that $r_1(U)\cap r_2(U_{\tilde{z}})=\emptyset$. 

Write $C$ for the complement of $U'\coloneqq \tilde{U}\cup\bigcup_{\tilde{z}\in s_2^{-1}(x)\backslash\{\tilde{y}\}}U_{\tilde{z}}$ in $Z_2$. Proper continuous maps to locally compact Hausdorff spaces are closed, so $s_2(C)$ is closed in $X$. Then $V\coloneqq X\backslash s_2(C)$ is an open neighborhood of $x$ such that $s_2^{-1}(V)$ is contained in $U'$. We shrink $U$ such that $s_1(U)$ is contained in $V$. 

Now take $z\in U$. Then $s_2^{-1}(s_1(z))$ is contained in $U'$. Hence the unique $\tilde{z}\in s_2^{-1}(s_1(z))$ such that $r_1(z)=r_2(\tilde{z})$ must lie in $\tilde{U}$. Thus $\phi(z)=\tilde{z}$ lies in $\tilde{U}$. Since $z\in U$ was arbitrary we obtain that $\phi$ maps $U$ into $\tilde{U}$. This shows that $\phi$ is continuous.  A continuous bijective map between compact Hausdorff spaces is a homeomorphism. 
\end{proof}

\begin{proposition}\label{prop:local iso of modules}
Let $\E$ and $\F$ be two proper nondegenerate full $\sigma$-finitely generated and $\sigma$-projective $C^*$-correspondences from $C_0(Y)$ to $C_0(X)$. We denote $C^*(\im\iota_\E,\im\varphi_\E)$ by $D_\E$. Write $r_\E$ and $s_\E$ for the Gelfand duals of $\varphi_\E:C_0(Y)\to D_\E$ and $\iota_\E:C_0(X)\to D_\E$, respectively. A similar notation is adopted for $\F$. Then the following are equivalent: 
\begin{enumerate}
\item We have $D_\E\cong_{\iota,\varphi}D_\F$.
\item We have $r_\E((s_\E)^{-1}(x))=r_\F((s_\F)^{-1}(x))$ for all $x\in X$.
\end{enumerate}
If the isomorphism from (1) exists, then it is unique.
\end{proposition}
\begin{proof}
That (1) implies (2) follows immediately from Definition \ref{def:iso of subalgebras}.

We now proof that (2) implies (1). All we have to show is that $r_\E(y)\neq r_\E(z)$ for all distinct $y$ and $z$ in $s_\E^{-1}(x)$, and that the same holds for $r_\F$, since then the statement of (1) as well as the uniqueness follow from Lemma \ref{lem:range map gives iso of different coverings}. To ease notation we drop the subscripts, since all statements hold regardless of whether we consider $\E$ or $\F$. 

Take distinct points $y$ and $z$ in $s^{-1}(x)$. Then there exists $d\in D_x$ such that $d(y)=0$ and $d(z)=1$. The fiber $D_x$ equals $\im\varphi_x$. Hence there exists $f\in C_0(Y)$ such that $\varphi(f)_x=d$. Now we have \[f(r(y))=\varphi(f)(y)=\varphi(f)_x(y)=d(y)=1,\] and in the same way we can show that $f(r(z))=0$. Thus $r(y)\neq r(z)$. This concludes the proof. 
\end{proof}
\begin{remark}\label{rem:left action spectrum}
Recall once more that by the proof of Proposition \ref{prop:correspondence to atlas}, the eigenvalues of $\varphi_\E(f)_x$ are given by $\{f(r_\E(y)):y\in s_\E^{-1}(x)\}=f(r_\E(s_\E^{-1}(x)))$. Hence Proposition \ref{prop:local iso of modules} shows that the subalgebra $C^*(\im\iota_\E,\im\varphi_\E)$ is completely determined by the eigenvalues of $\varphi(f)$ at each point, and for all $f\in C_0(Y)$. 
\end{remark}
\subsection{Local isomorphism of $C^*$-correspondences}
Recall that we call two $C^*$-correspondences $\E$ and $\F$ from $C_0(Y)$ to $C_0(X)$ isomorphic if there exists a linear inner-product preserving bijection $T$ from $\E$ to $\F$ such that $T(\xi f)=T(\xi)f$ and $T(\varphi_\E(g)\xi)=\varphi_\F(g)T(\xi)$ for all $\xi\in\E$, $f\in C_0(X)$ and $g\in C_0(Y)$. 
\begin{definition}\label{def:local isomorphism}
Let $\E$ and $\F$ be $C^*$-correspondences from $C_0(Y)$ to $C_0(X)$. Then we call $\E$ and $\F$ \emph{locally isomorphic}, and write $\E\cong_{\mathrm{loc}}\F$, if any point $x\in X$ has an open neighborhood $U$ such that the $C^*$-correspondences $\E_U:C_0(X)\to C_0(U)$ and $\F_U:C_0(X)\to C_0(U)$ are isomorphic. Here $\E_U\coloneqq \E C_0(U)$ is defined by restricting the right Hilbert module structure of $\E$ to $C_0(U)$, but leaving the left action unchanged, and similarly for $\F_U$. 
\end{definition}
\begin{proposition}\label{prop:local iso implies iso of abelian subalgebras}
Assume that $\E$ and $\F$ are locally isomorphic. Then 
\[D_\E\cong_{\iota,\varphi}D_\F\quad\text{and}\quad D_\E''\cong_{\iota,\varphi}D_{\F}''\]
where we again write $D_\E$ and $D_\F$ for $C^*(\im\iota_\E,\im\varphi_\E)$ and $C^*(\im\iota_\F,\im\varphi_\F)$, respectively.
\end{proposition}
\begin{proof}
Local isomorphism of $\E$ and $\F$ implies that $(\varphi_\E)_x(f)$ and $(\varphi_\F)_x(f)$ are similar matrices for all $f\in C_0(Y)$ and all $x\in X$, and thus have the same spectrum. Taking into account Remark \ref{rem:left action spectrum} we obtain that Item (2) from Proposition \ref{prop:local iso of modules} is fulfilled. This shows that local isomorphism of $\E$ and $\F$ implies $D_\E\cong_{\iota,\varphi}D_\F$.

Write $Z_\E$, $\tilde{Z}_\E$, $Z_\F$, and $\tilde{Z}_\F$ for the Gelfand duals of $D_\E$, $D_\E''$, $D_\F$ and $D_\F''$, respectively. Denote the Gelfand duals of $\iota_\E$ as maps into $D_\E$ and $D_\E''$ by $s_\E$ and $\tilde{s}_\E$, respectively. In the same way we define $s_\F$, and $\tilde{s}_\F$. Also define $\tilde{r}_\E:\tilde{Z}_\E\to Y$ and $\tilde{r}_\F:\tilde{Z}_\F\to Y$ as the Gelfand duals of the left actions $\varphi_\E$ and $\varphi_\F$, respectively. The inclusion of $D_\E$ into $D_\E''$ yields a surjection of $\tilde{Z}_\E$ onto $Z_\E$ making the diagram 
\[\begin{tikzcd}
	{\tilde{Z}_\mathcal{E}} && {Z_\mathcal{E}} \\
	& X
	\arrow[two heads, from=1-1, to=1-3]
	\arrow["{\tilde{s}_\mathcal{E}}"', from=1-1, to=2-2]
	\arrow["{s_\mathcal{E}}", from=1-3, to=2-2]
\end{tikzcd}\]
commute.

We know from Lemma \ref{lem:double commutant} that $\iota_\E(C_0(X\backslash B))D_\E$ and $\iota_\E(C_0(X\backslash B))D_\E''$ are equal as subalgebras of $\K(\E)$. Since we have $\iota_\E(C_0(X\backslash B))D_\E=C_0(s_\E^{-1}(X\backslash B))$ and $\iota_\E(C_0(X\backslash B))D_\E''=C_0(\tilde{s}_\E^{-1}(X\backslash B))$ this shows that the surjection from the above diagram restricts to a homeomorphism from $\tilde{s}_\E^{-1}(X\backslash B)$ to $s_\E^{-1}(X\backslash B)$. Of course all of this holds in exactly the same way if we replace $\E$ with $\F$. 

Since $\E$ and $\F$ are locally isomorphic, for every point $x$ in $X$ there exist an open neighborhood $U$ of $x$ and a homeomorphism $\phi_U$ from $\tilde{s}_\E^{-1}(U)$ to $\tilde{s}_\F^{-1}(U)$ such that $\tilde{s}_\E=\tilde{s}_\F\circ\phi_U$ as well as $\tilde{r}_\E=\tilde{r}_\F\circ\phi_U$. When restricted to $\tilde{s}_\E^{-1}(U\backslash B)$ it is uniquely determined, since it induces a uniquely determined homeomorphism from $s_\E^{-1}(U\backslash B)$ to $s_\F^{-1}(U\backslash B)$ by the discussion above. Since $U\backslash B$ is dense in $U$, it follows that $\phi_U$ is the unique homeomorphism from $\tilde{s}_\E^{-1}(U)$ to $\tilde{s}_\F^{-1}(U)$ such that $\tilde{s}_\E=\tilde{s}_\F\circ\phi$ and $\tilde{r}_\E=\tilde{r}_\F\circ\phi$. Uniqueness yields a homeomorphism $\phi$ from $\tilde{Z}_\E$ to $\tilde{Z}_\F$ such that $\tilde{s}_\E=\tilde{s}_\F\circ\phi$ as well as $\tilde{r}_\E=\tilde{r}_\F\circ\phi$. This shows $D_\E''\cong_{\iota,\varphi}D_{\F}''$.
\end{proof}
The next two examples show that in general, the implication from Proposition \ref{prop:local iso implies iso of abelian subalgebras} is not an equivalence. Even if $C^*(\im\iota_\E,\im\varphi_\E)''$ and $C^*(\im\iota_\F,\im\varphi_\F)''$ are isomorphic in the sense of Definition \ref{def:iso of subalgebras}, this is not sufficient to conclude that $\E$ and $\F$ are locally isomorphic. 

\begin{example}\label{ex:trivial covering}
Let $X$ be any compact Hausdorff space, $Z=X\times\{1,2,3\}$ and $Y=\{1,2\}$. Let $s:Z\to X$ be the trivial 3-sheeted covering. Define two range maps 
\[r_1(x,i)=\begin{cases}1\quad&\text{for }i\in\{1,2\},\\2\quad&\text{for }i=3, \end{cases}\quad\text{and}\quad r_2(x,i)=\begin{cases}1\quad&\text{for }i=1,\\2\quad&\text{for }i\in\{2,3\}. \end{cases}\]
Write $\E$ for $\Gamma_0(Z,r_1,s)$ and $\F$ for $\Gamma_0(Z,r_2,s)$. We adopt the notation from the proof of Proposition \ref{prop:local iso implies iso of abelian subalgebras}. Then $Z_\E=\tilde{Z}_\E=Z_\F=\tilde{Z}_F=X\times\{1,2\}$ and $\tilde{s}_\E=s_\E=\tilde{s}_F=s_\F$ is the trivial 2-sheeted covering of $X$. The range maps send the first copy of $X$ in this 2-sheeted covering to $1\in Y$, and the second copy to $2$. However, $\E$ and $\F$ are not locally isomorphic. 
\end{example}
Example \ref{ex:trivial covering} seems to suggest that we can capture local isomorphism of the $C^*$-correspondences by keeping track of the dimensions of the projections $p_y$ as in Lemma \ref{lem:structure of module from abelian subalgebra}. In Corollary \ref{cor:local iso for homogeneous subalgebras} we will show that this is indeed the case if $C^*(\im\iota,\im\varphi)$ or its double commutant have constant rank. However, the next example shows that it is not always true. 
\begin{example}
This example is a two-dimensional version of Example \ref{ex:different range maps}. Let $X$ be the square $[-1,1]^2$, let $Z$ be two disjoint copies of $X$, and let $s:Z\to X$ be the trivial 2-sheeted covering. Define the two range maps 
\begin{align*}r_1(x,y,i)&=\begin{cases}y|x^2-y^2|\quad&\text{for }i=1,|x|<|y|,\\0\quad&\text{for }i=1,|y|\leq|x|,\\-y|x^2-y^2|\quad&\text{for }i=2,|x|<|y|,\\0\quad&\text{for }i=2,|y|\leq|x|, \end{cases}\quad\text{and}\\ r_1(x,y,i)&=\begin{cases}|y||x^2-y^2|\quad&\text{for }i=1,|x|<|y|,\\0\quad&\text{for }i=1,|y|\leq|x|,\\-|y||x^2-y^2|\quad&\text{for }i=2,|x|<|y|,\\0\quad&\text{for }i=2,|y|\leq|x|. \end{cases}\end{align*}
As before we write $\E$ for $\Gamma_0(Z,r_1,s)$ and $\F$ for $\Gamma_0(Z,r_2,s)$. Then 
\[D_\E=D_\F=\{f\in C(X,D_2):f(x,y)\in\C 1_{M_2}\text{ for }|x|\geq|y|\}\]and hence 
\[D_\E'= D_\F'=\{f\in C(X,M_2):f(x,y)\in D_2\text{ for }|x|\leq|y|\}.\]
Therefore we have $D_\E''=D_\F''=D_\E=D_\F$. We can define an isomorphism $\Phi$ from $D_\E$ to $D_\F$ in the sense of Definition \ref{def:iso of subalgebras} by the following: For $|y|>|x|$ and $y<0$ let $\Phi(f)(x,y)=\sigma(f(x,y))$, where $\sigma$ is the map on $D_2$ permuting the diagonal entries. Set $\Phi(f)(x,y)=f(x,y)$ otherwise. Thus we cannot distinguish the two graphs. Yet $\Gamma(Z,r_1,s)$ and $\Gamma(Z,r_2,s)$ are not locally isomorphic, which one can see by looking at the point $(0,0)$. Note however that there is no isomorphism from $D_\E'$ to $D_\F'$ intertwining the left actions. This can be shown by considering the eigenvectors of the left action in each fiber. In fact we can reconstruct the graphs, and thus distinguish them, by considering the commutants instead of double commutants of $D_\E$ and $D_\F$. We suspect that this observation generalizes, and that local isomorphism of $C^*$-correspondences $\E$ and $\F$ is equivalent to isomorphism of $C^*(\im\iota_\E,\im\varphi_\E)'$ and $C^*(\im\iota_\F,\im\varphi_\F)'$ in the sense of Definition \ref{def:iso of subalgebras}.
\end{example}

We now investigate what happens if the abelian subalgebra $D$ of $\K(\E)$ under consideration has constant rank, meaning that the dimension of $D_x$ is the same for all $x\in X$. It turns out that one can characterize isomorphism of two $C^*$-correspondences $\E$ and $\F$ in terms of isomorphism of such subalgebras $D_\E$ and $D_F$ of $\K(\E)$ and $\K(\F)$, respectively, with an additional assumption concerning vector bundles over the spectra $\widehat{D}_E$ and $\widehat{D}_\F$.

Let $D$ be an abelian subalgebra of $\K(\E)$ with constant rank, containing $C_0(X)$. It follows from Section \ref{sect:cartans and corrs} that $\E^D$, namely $\E$ regarded as a $D$-module, is $\sigma$-finitely generated and $\sigma$-projective. We can construct a $D$-valued inner product $\langle\cdot,\cdot\rangle_\E^D$ turning $\E^D$ into a Hilbert $D$-module.  We write $\V_\E^D$ for the vector bundle over $\widehat{D}_\E$ associated to $\E^D$. By slight abuse of notation we write $\iota_\E$ for both the inclusion of $C_0(X)$ into $\K(\E)$ and into $D$. The Gelfand dual of $\iota_\E:C_0(X)\to D$ is denoted by $s_\E:\widehat{D}\to X$.

\begin{proposition}\label{prop:homogeneous subalgebra isomorphism}
Let $\E$ and $\F$ be proper nondegenerate full $\sigma$-finitely generated and $\sigma$-projective $C^*$-correspondences from $C_0(Y)$ to $C_0(X)$. Let $D_\F$ and $D_\E$ be constant-rank abelian subalgebras of $\K(\E)$ and $\K(\F)$, respectively. Assume that the images of $\varphi_\E$ and $\iota_\E$ as well as $\varphi_\F$ and $\iota_\F$ are contained in $D_\E$ and $D_\F$, respectively. Furthermore, assume that both of the following two conditions are fulfilled: 
\begin{enumerate}
\item We have $D_\E\cong_{\iota,\varphi} D_\F$. That is, there exists a $\ast$-isomorphism $\Pi$ from $D_\E$ to $D_\F$ such that $\Pi\circ\varphi_\E=\varphi_\F$ and $\Pi\circ\iota_\E=\iota_F$.
\item The vector bundle $\V_\E^D$ is isomorphic to the pullback of the vector bundle $\V_\F^D$ along the map $\widehat{\Pi}^{-1}:\widehat{D}_\E\to\widehat{D}_\F$. 
\end{enumerate}
Then $\E$ and $\F$ are isomorphic. 
\end{proposition}
\begin{proof}
What Conditions (1) and (2) are saying is that there is an isomorphism $(\Pi,T)$ of Hilbert modules from $\E^D$ to $\F^D$, with the additional condition that $\Pi\circ\varphi_\E=\varphi_\F$ and $\Pi\circ\iota_\E=\iota_F$ holds. Here $\Pi:D_\E\to D_F$ is the $\ast$-isomorphism from condition (1), and $T:\E^D\to\F^D$ corresponds to the vector bundle isomorphism from condition (2) via the Serre-Swan theorem. 

Since as sets $\E^D$ and $\F^D$ are equal to $\E$ and $\F$, respectively, we can regard $T$ as a map from $\E$ to $\F$. Take $\xi\in \E$ and $f\in C_0(X)$, and calculate
\begin{align*}
T(\xi)f=\iota_\F(f)T(\xi)=\Pi(\iota_\E(f))T(\xi)=T(\iota_\E(f)\xi)=T(\xi f).
\end{align*}
This shows that $T$ intertwines the right multiplication of $\E$ and $\F$. Now we show that it preserves the inner product. Recall that by definition of the $D$-valued inner product on $\E^D$ we have 
\[\langle\xi,\eta\rangle_\E(x)=\sum_{y\in s^{-1}(x)}\langle\xi,\eta\rangle_\E^D(y)
\]
for all $\xi,\eta\in\E$, where $s=\widehat{\iota}_\E$ is the branched covering from Lemma \ref{lem:abelian C_0(X) algebra branched covering}. Note that we tacitly identify $\E$ with $\E^D$ as sets. Of course a similar formula holds for the $D$-valued inner product on $\F^D$. With this in mind, calculate
\begin{align*}
\langle T\xi,T\eta\rangle_\F(x)=\sum_{y\in s^{-1}(x)}\langle T \xi,T\eta\rangle_\F^D(y)=\sum_{y\in s^{-1}(x)}\langle \xi,\eta\rangle_\E^D(y)=\langle\xi,\eta\rangle_\E(x). 
\end{align*}
This shows that $T$ is an isomorphism of Hilbert $C_0(X)$-modules from $\E$ to $\F$. Lastly we have
\[\varphi_\F(f)T(\xi)=\Pi(\varphi_\E(f))T(\xi)=T(\varphi_\E(f)\xi)\]
for all $\xi\in \E$ and $f\in C_0(Y)$, which shows that $T$ intertwines the left actions $\varphi_\E$ and $\varphi_\F$, and hence is an isomorphism of $C^*$-correspondences. This concludes the proof. 
\end{proof}
We can weaken the assumptions of Proposition \ref{prop:homogeneous subalgebra isomorphism} to obtain local isomorphism of $\E$ and $\F$ instead of isomorphism. More precisely, instead of demanding that the vector bundles $\V_\E^D$ and $\V_\F^D$ be isomorphic, we only want them to have the same rank over each point of $\widehat{D}_\E\cong\widehat{D}_\F$. 
\begin{corollary}\label{cor:local iso for homogeneous subalgebras}
Let $\E$, $\F$, $D_\E$, $D_\F$ and $\Pi$ be as in Proposition \ref{prop:homogeneous subalgebra isomorphism}, except that we do not assume condition (2). Instead, we demand that the rank of $\V_\E^D$ is the same as the rank of $(\widehat{\Pi}^{-1})^*(\V_\F^D)$ over each $z\in\widehat{D}_\E$. Then $\E$ and $\F$ are locally isomorphic in the sense of Definition \ref{def:local isomorphism}.
\end{corollary}
\begin{proof}
Take $x\in X$, and let $K$ be a closed neighborhood of $x$ such that $\E^D$ is trivial over $s_\E^{-1}(K)$, and $\F^D$ is trivial over $s_\F^{-1}(K)$. Then the vector bundle associated to $\E^D_{s_\E^{-1}(K)}$ and the pullback of the vector bundle associated to $\F^D_{s_\F^{-1}(K)}$ along $\widehat{\Pi}^{-1}$ are two trivial vector bundles which have the same rank over each point of $s_\E^{-1}(K)$ . Thus they are isomorphic. Proposition \ref{prop:homogeneous subalgebra isomorphism} yields that $\E_K$ and $\F_K$ are isomorphic as $C^*$-correspondences from $C_0(X)$ to $C_0(K)$. Since $x$ was arbitrary we obtain that $\E$ and $\F$ are locally isomorphic. 
\end{proof}

\subsection{Local conjugacy of topological correspondences}

The following definition of local conjugacy for twisted topological correspondences is inspired by the similar definition for topological graphs, see \cite[Definition 4.3]{davidsonRoydor:2011}. The difference is that we do not allow for homeomorphisms on the base spaces $Y$ and $X$, see Remark \ref{rem:base preserving isos}.
\begin{definition}\label{def:local conjugacy}
We call two topological $Y$-$X$-correspondences $(Z,r,s)$ and $(Z',r',s')$ \emph{locally conjugate} if every point $x\in X$ has an open neighborhood $U$ such that the $Y$-$U$-correspondences $(Z_U,r_U,s_U)$ and $(Z_U',r_U',s_U')$ are isomorphic. Here we write $Z_U$ for $s^{-1}(U)$, and $s_U:Z_U\to U$ and $r_U:Z_U\to Y$ for the restrictions of $s$ and $r$, respectively. The same goes for $Z_U'$, $r_U'$ and $s_U'$. 

Two twisted topological $Y$-$X$-correspondences $(Z,r,s,\calL)$ and $(Z',r',s',\calL')$ are called locally conjugate if the underlying topological correspondences $(Z,r,s)$ and $(Z',r',s')$ are. 
\end{definition}

The following proposition is a generalization of \cite[Theorem 4.5]{frausino2023}. Note that the proof is the same as in \cite{frausino2023}, but formulated in the language developed in the Sections \ref{sect:correspondences} and \ref{sect:perspective of the compacts}. First we need a lemma.
\begin{lemma}[{\cite[Lemma 4.2]{frausino2023}}]\label{lem:bases}
    Let $(e_1,\dots,e_n)$ be the standard basis of $\C^n$. If $(x_1,\dots,x_n)$ is another basis of $\C^n$, then there exists a permutation $\sigma\in S_n$ such that for all $k=1,\dots,n$ the vectors $x_k$ and $e_{\sigma(k)}$ are not orthogonal.
\end{lemma}

\begin{proposition}\label{prop:local conjugacy}
Let $(Z,r,s,\calL)$ and $(Z',r',s',\calL')$ be proper twisted topological $Y$-$X$ correspondences, for locally compact Hausdorff spaces $X$ and $Y$. Then the following are equivalent:
\begin{itemize}
\item[(a)] The correspondences $(Z,r,s,\calL)$ and $(Z',r',s',\calL')$ are locally conjugate in the sense of Definition \ref{def:local conjugacy}.
\item[(b)]The $C^*$-correspondences $\Gamma_0(Z,r,s,\calL)$ and $\Gamma_0(Z',r',s',\calL')$ are locally isomorphic in the sense of Definition \ref{def:local isomorphism}.
\end{itemize}
\end{proposition}
\begin{proof}
    That (a) implies (b) is clear. 

    To see that (b) implies (a), take $x\in X$. Write $\E\coloneqq\Gamma_0(Z,r,s,\calL)$ and $\E'\coloneqq\Gamma_0(Z',r',s',\calL')$. Let $U$ be a neighborhood of $x$ such that there exists an isomorphism $\Phi$ from $\E_U$ to $\E'_U$. We can choose $U$ small enough so that $s_U:s^{-1}(U)\to U$ and $s'_U:(s')^{-1}(U)\to U$ are trivial coverings. Take charts associated to the graphs $(s^{-1}(U),r_U,s_U,\calL_{s^{-1}(U)})$ and $((s')^{-1}(U),r'_U,s'_U,\calL'_{(s')^{-1}(U)})$ as in Theorem \ref{thm:corr and atlas one to one}. Let $\xi_k\in\E$ and $\xi_k'\in\E'$ for $k=1,\dots,n_x$ be sections supported on $U$ constructed from those charts as in Remark \ref{rem:frame}. By construction, $(\xi_1(x),\dots,\xi_{n_x}(x))$ is an orthogonal basis for $\E_x\cong\C^{n_x}$. Since  $\Phi$ is inner-product preserving, $(\Phi(\xi_1)(x),\dots,\Phi(\xi_{n_x})(x))$ is a basis of $\E_x'\cong\C^{n_x}$. By Lemma \ref{lem:bases} there exists a permutation $\sigma\in S_{n_x}$ such that for all $k=1,\dots,n_x$, the vectors $\xi_k'(x)$ and $\Phi(\xi_{\sigma(k)})(x)$ are not orthogonal. 
    
Since $\xi_k$ was constructed from a chart diagonalizing the left action, $\xi_k(x)$ is an eigenvector of $\varphi_\E(f)_x$ for all $f\in C_0(Y)$. The same holds for $\xi_k'(x)$ and $\varphi_{\E'}(f)_x$. Note that $\Phi$ is an isomorphism of $C^*$-correspondences, and hence $\Phi(\xi_{\sigma(k)})(x)$ is an eigenvector of $\varphi_{\E'}(f)_x$ as well. Since $\xi_k'(x)$ and $\Phi(\xi_{\sigma(k)})(x)$ are not orthogonal, they have to belong to the same eigenspace. If we write $\lambda_k(x)$ for the eigenvalue of $\varphi_\E(f)_x$ associated to the eigenvector $\xi_k(x)$, and accordingly $\lambda_k'(x)$, then we have shown $\lambda_k'(x)=\lambda_{\sigma(k)}(x)$. 

There exists a closed neighborhood $K$ of $x$ contained in $U$ such that $\xi_k'(y)$ and $\Phi(\xi_{\sigma(k)})(y)$ are not orthogonal for all $k=1,\dots,n_y$ and all $y\in K$. The argument in the previous paragraph shows $\lambda_k'(y)=\lambda_{\sigma(k)}(y)$. Thus if we define a map $\Psi$ from $\E_K$ to $\E'_K$ by sending $(\xi_{\sigma(k)})_K$ to $(\xi_k')_K$ for all $k=1,\dots,n_x$, then $\Psi$ intertwines the left actions. Since $\{\xi_k(y)\}_k$ is an orthogonal basis for $\E_y$ for every $y\in K$, and the same holds for $\{\xi_k'(y)\}_k$ and $\E'_y$, we see that $\Psi$ is an isomorphism of $C^*$-correspondences from $\E_K$ to $\E'_K$. Now employ Proposition \ref{prop:isos for different Cstar correspondences} to see that $(Z,r,s,\calL)$ and $(Z',r',s',\calL')$ are locally conjugate. 
\end{proof}

\begin{remark}\label{rem:base preserving isos}
    The reader might have noticed that since our notions of isomorphism of $C^*$-correspondences as well as local conjugacy of topological graphs are stricter than what is used in the literature, Proposition \ref{prop:local conjugacy} does not seem to be an honest generalization of \cite[Theorem 4.5]{frausino2023}. Usually one would call two $C^*$-correspondences $\E$ and $\F$ over $A$ and $B$, respectively, isomorphic if there exist a linear map $T$ from $\F$ to $\E$ and a $\ast$-homomorphism $\pi$ from $A$ to $B$, such that 
    \[\langle T(\xi),T(\eta)\rangle_\F=\pi(\langle\xi,\eta\rangle_\E),\quad T(\xi a)=T(\xi)\pi(a),\quad T(\varphi_\E(a)\xi)=\varphi_\F(\pi(a))T(\xi),\]
    for all $\xi,\eta\in\E$ and $a\in A$. In a similar way, the definition of local conjugacy in \cite{davidsonRoydor:2011} allows for the topological graphs to be over different vertex spaces, and involves a homeomorphism between these spaces. 

    However, if $\E$ and $\F$ are isomorphic via $(T,\pi)$, then we can define a new $C^*$-correspondence $\tilde{\F}$ over $A$ by 
    \[\langle\xi,\eta\rangle_{\tilde{\F}}\coloneqq \pi^{-1}(\langle\xi,\eta\rangle_\F),\quad \xi\cdot_{\tilde{\F}} a\coloneqq \xi\cdot \pi(a),\quad \varphi_{\Tilde{\F}}(a)\xi\coloneqq \varphi_\F(\pi(a))\xi,\]
    where $\xi,\eta\in\F$ and $a\in A$. Then $\E$ and $\tilde{\F}$ are isomorphic in the sense that we have used in this paper. Similarly, if given a topological graph and a homeomorphism between its vertex space and some other space, one can pull back the graph along the homeomorphism. Therefore Proposition \ref{prop:local conjugacy} does in fact imply \cite[Theorem 4.5]{frausino2023}.

\end{remark}

In \cite{frausino2023} it was asked what information is necessary in addition to local conjugacy to conclude isomorphism of the $C^*$-correspondences associated to two topological graphs. That local conjugacy alone is not sufficient was shown in \cite{bilich2025obstructionisomorphismtensoralgebras}, and it also follows from the example in Section \ref{sect:covering with nontrivial vector bundle}. We can answer the question in the special case that one of the algebras $C^*(\im\iota_\E,\im\varphi_\E)$ or $C^*(\im\iota_\E,\im\varphi_\E)''$ from the previous section has constant rank. 

Let $(Z,r,s,\calL)$ and $(Z',r',s',\calL')$ be proper twisted topological $Y$-$X$-correspondences. Write $\E$ and $\E'$ for the $C^*$-correspondences $\Gamma_0(Z,r,s,\calL)$ and $\Gamma_0(Z',r',s',\calL')$, respectively. Assume without loss of generality that $s$ and $s'$ are surjective, or equivalently that $\E$ and $\E'$ are right full. Also assume that either $C^*(\im\iota_\E,\im\varphi_\E)$ and $C^*(\im\iota_{\E'},\im\varphi_{\E'})$ or their double commutants have constant rank. Write $D_\E$ and $D_{\E'}$ for the respective algebras. Adopt the notation of Proposition \ref{prop:homogeneous subalgebra isomorphism} for the vector bundles $\V_\E^D$ and $\V_{\E'}^D$.

If $(Z,r,s,\calL)$ and $(Z',r',s',\calL')$ are locally conjugate, then by Proposition \ref{prop:local conjugacy} the $C^*$-correspondences $\E$ and $\E'$ are locally isomorphic. Hence, by Proposition \ref{prop:local iso implies iso of abelian subalgebras} we have $D_\E\cong_{\iota,\varphi}D_{\E'}$. Write $\Pi$ for this isomorphism.

\begin{corollary}\label{cor:iso and local conjugacy}
The $C^*$-correspondences $\E$ and $\E'$ are isomorphic if and only if both of the following conditions hold:
\begin{enumerate}
\item[(a)] The twisted $Y$-$X$-correspondences $(Z,r,s,\calL)$ and $(Z',r',s',\calL')$ are locally conjugate. 
\item[(b)] The vector bundle $\V_\E^D$ is isomorphic to the pullback of the vector bundle $\V_{\E'}^D$ along the map $\widehat{\Pi}^{-1}:\widehat{D}_\E\to\widehat{D}_{\E'}$. 
\end{enumerate}
\end{corollary}
\begin{proof}
If $\E$ and $\E'$ are isomorphic, then (a) holds because of Proposition \ref{prop:local conjugacy}. Condition (b) holds because of how $D_{\E}$, $D_{\E'}$ as well as $\V_\E^D$ and $\V_{\E'}^D$ are defined.

If we assume both (a) and (b), then it follows from Proposition \ref{prop:homogeneous subalgebra isomorphism} that $\E$ and $\E'$ are isomorphic.
\end{proof}
In \cite{frausino2023} the authors gave an example of two nonisomorphic topological graphs whose $C^*$-correspondences are isomorphic. The example consisted of the two 2-sheeted coverings of $S^1$, whose associated vector bundles are trivial as we have seen in Section \ref{sec:twisted coverings over spheres}. Thus in this example, the fact that the topological graphs are not isomorphic stems from the fact that the underlying coverings are different. One might wonder whether any two graphs with isomorphic $C^*$-correspondences and the same underlying covering must be isomorphic. To conclude this section, we give an example showing that this is not the case.
\begin{example}
    This is a variant of Example \ref{ex:different range maps}.
    Take $X=Y=[-1,1]$, $Z=[-1,1]\times\{1,2\}$ and let $s$ be the trivial covering $s(x,i)=x$. Let $f\in C([-1,1])$ be the function given by 
    \[f(x)=\begin{cases}x+\frac{1}{2}\quad&\text{for }x<-\frac{1}{2},\\ 0\quad&\text{for }-\frac{1}{2}<x<\frac{1}{2},\\x-\frac{1}{2}\quad&\text{for }x>\frac{1}{2}.\end{cases}\]
     We define two different range maps,
    \[r_1(x,i)=\begin{cases}f\quad&\text{for }i=1,\\-f\quad&\text{for }i=2, \end{cases}\quad\text{and}\quad r_2(x,i)=\begin{cases}|f|\quad&\text{for }i=1,\\-|f|\quad&\text{for }i=2. \end{cases}\]
    The resulting topological graphs are not isomorphic, but have isomorphic $C^*$-correspondences.
\end{example}

\appendix

\section{Covering spaces and principal bundles}\label{appendix A}
 Let $\mathcal{P}=(P,p,X)$ be a principal $S_n$-bundle over $X$. Form the fiber bundle $\mathcal{Z}=(Z,p_Z,X)$ with fiber $\{1,\dots,n\}$ associated to $\mathcal{P}$ under the permutation action of $S_n$ on $\{1,\dots,n\}$. 
\begin{lemma}\label{lem:bundle to covering}
    The bundle projection $p_Z:Z\to X$ is an $n$-sheeted covering.
\end{lemma}
\begin{proof}
    Take $x\in X$. The principal bundle $\mathcal{P}$ is locally trivial, and and therefore $\mathcal{Z}$ is as well. Hence exist an open neighborhood $U$ of $x$ and a homeomorphism $ h$ from $U\times\{1,\dots,n\}$ to $p_Z^{-1}(U)$ such that $p_Z\circ  h=p_1$. This is exactly the definition of an $n$-sheeted covering map. 
\end{proof}
\begin{lemma}\label{lem:covering to bundle}
    Let $s:Z\to X$ be an $n$-sheeted covering map. Then we can construct a principal $S_n$-bundle such that carrying out the construction from Lemma \ref{lem:bundle to covering} gives back $s$. 
\end{lemma}
\begin{proof}
    We can find a finite open cover $\{U_i\}_{i\in I}$ of $X$ and homeomorphisms $h_i$ from $U_i\times\{1,\dots,n\}$ to $s^{-1}(U_i)$ such that $s\circ h=p_1$ for all $i\in I$. Now let $i,j\in I$ be such that the intersection of $U_i$ and $U_j$ is nonempty, and take $x\in U_i\cap U_j$. Then $h_i^{-1}\circ h_j$ is a bijection on $p_1^{-1}(x)$, and hence determines a permutation $\sigma_{ij}(x)\in S_n$. We write $\sigma_{ij}$ for the map from $U_i\cap U_j$ to $S_n$ sending $x$ to $\sigma_{ij}(x)$.

    Now take $y\in s^{-1}(x)$. Then $ h_j^{-1}(y)=(x,k)$ and $ h_i^{-1}(y)=(x,\sigma_{ij}(x)k)$ for some $k\in\{1,\dots,n\}$. We obtain an open neighborhood of $y$ by setting $V\coloneqq  h_j(U_j\times\{k\})\cap  h_i(U_i\times\{\sigma_{ij}(x)k\})$. The image of any $z\in V$ under $ h_j^{-1}$ is $(s(z),k)$, and under $ h_i^{-1}$ it is $(s(z),\sigma_{ij}(x)k)$. Therefore $\sigma_{ij}(s(z))k=\sigma_{ij}(x)k$ for all $z\in V$. Since we could have chosen any element from $s^{-1}(x)$ as $y$, and since there are only finitely many of them, this shows that there exists a neighborhood of $x$ on which $\sigma_{ij}$ is constant. Hence $\sigma_{ij}$ is locally constant, and in particular continuous. 

    It is clear that the cocycle identity $\sigma_{ij}\sigma_{jl}=\sigma_{il}$ holds. We know from section \ref{sct:principal bundles and fiber bundles} that this data specifies a principal $S_n$-bundle such that the total space of the associated fiber bundle with fiber $\{1,\dots,n\}$ is given by $Z$, with bundle projection $s$. We are done. 
\end{proof}
Let $s:Z\to X$ and $s':Z'\to X$ be two $n$-sheeted covering maps of $X$. A map $\rho$ from $Z$ to $Z'$ is called \emph{base-preserving} if $s'\circ\rho=s$. An isomorphism of coverings is defined to be a base-preserving homeomorphism. 
\begin{proposition}\label{prop:covering spaces and principal bundles}
The constructions from Lemmas \ref{lem:bundle to covering} and \ref{lem:covering to bundle} give a bijection between isomorphism classes of $n$-sheeted coverings of $X$ and isomorphism classes of principal $S_n$-bundles over $X$. 
 \end{proposition}
\begin{proof}
 What is left to prove is that $\mathcal{P}$ and $\mathcal{P}'$ are isomorphic as principal $S_n$-bundles if and only if the associated covering spaces are isomorphic. One direction is clear, since an isomorphism from $\mathcal{P}$ to $\mathcal{P}'$ will induce an isomorphism between the associated fiber bundles $\mathcal{Z}$ and $\mathcal{Z}'$. Such an isomorphism is in particular a base-preserving homeomorphism between $Z$ and $Z'$, and thus an isomorphism of covering maps. 
 
 On the other hand, assume that we have an isomorphism $\rho$ of covering maps from $s:Z\to X$ to $s':Z'\to X$. We can choose a finite cover $\{U_i\}$ of $X$ such that both $s$ and $s'$ are trivial over the $U_i$. Write $ h_i:U_i\times\{1,\dots,n\}\to s^{-1}(U_i)$ for the homeomorphisms that trivialize $s$, and similarly define $ h_i'$. Since $\rho$ is by definition base-preserving, we obtain continuous maps $r_i:U_i\to S_n$ for every $i\in I$ such that the diagram 
\[\begin{tikzcd}
	{s^{-1}(U_i)} & {U_i\times\{1,\dots,n\}} \\
	{(s')^{-1}(U_i)} & {U_i\times\{1,\dots,n\}}
	\arrow["\rho", from=1-1, to=2-1]
	\arrow["{h_i}"', from=1-2, to=1-1]
	\arrow["{r_i}", from=1-2, to=2-2]
	\arrow["{h_i'}"', from=2-2, to=2-1]
\end{tikzcd}\]
commutes, where by abuse of notation we also write $r_i$ for the map from $U_i\times\{1,\dots,n\}$ to itself sending $(x,k)$ to $(x,r_i(x)k)$. It follows from the diagram 
\[\begin{tikzcd}
	{U_i\cap U_j\times\{1,\dots,n\}} & {s^{-1}(U_i\cap U_j)} & {U_i\cap U_j\times\{1,\dots,n\}} \\
	{U_i\cap U_j\times\{1,\dots,n\}} & {(s')^{-1}(U_i\cap U_j)} & {U_i\cap U_j\times\{1,\dots,n\}}
	\arrow["{h_j}", from=1-1, to=1-2]
	\arrow["{r_j}", from=1-1, to=2-1]
	\arrow["\rho", from=1-2, to=2-2]
	\arrow["{h_i}"', from=1-3, to=1-2]
	\arrow["{r_i}", from=1-3, to=2-3]
	\arrow["{h_j'}", from=2-1, to=2-2]
	\arrow["{h_i'}"', from=2-3, to=2-2]
\end{tikzcd}\]
that for all $i,j\in I$ we have $\sigma_{ij}=r_i^{-1}\sigma_{ij}'r_j$. Hence the principal bundles constructed from $s$ and $s'$ are isomorphic by \cite[Theorem 2.7, Chapter 4]{Husemoller:1993}. This concludes the proof. 
 \end{proof}

\printbibliography

\end{document}